\documentclass[11pt,a4paper]{amsart}
\usepackage{amssymb,amsmath,mathrsfs}
\usepackage{times}

\usepackage[left=3.1 cm,right=3.1 cm,top=3.9cm,bottom=3.5cm]{geometry}

\usepackage{color}

\usepackage{bbm}
\usepackage{enumitem}

\newtheorem{theorem}{Theorem}[section]
\newtheorem{lemma}[theorem]{Lemma}
\newtheorem{proposition}[theorem]{Proposition}

\newtheorem{definition}[theorem]{Definition}

\newtheorem{remark}[theorem]{Remark}

\numberwithin{equation}{section}

\newcommand{\R}{\mathbb{R}}

\newcommand{\dd}{\mathrm{d}}

\newcommand{\CC}{\operatorname{C}}

\newcommand{\LL}{\operatorname{L}}

\newcommand{\WW}{\operatorname{W}}

\newcommand{\mylabel}[2]{#2\def\@currentlabel{#2}\label{#1}}

\def\Xint#1{\mathchoice
   {\XXint\displaystyle\textstyle{#1}}%
   {\XXint\textstyle\scriptstyle{#1}}%
   {\XXint\scriptstyle\scriptscriptstyle{#1}}%
   {\XXint\scriptscriptstyle\scriptscriptstyle{#1}}%
   \!\int}
\def\XXint#1#2#3{{\setbox0=\hbox{$#1{#2#3}{\int}$}
     \vcenter{\hbox{$#2#3$}}\kern-.5\wd0}}

\def\dashint{\Xint-}

\DeclareMathAlphabet{\mathpzc}{OT1}{pzc}{m}{it}

\begin{document}
\title{Partial regularity for local minimizers of variational integrals with lower order terms}

\author[J.~Campos Cordero]{Judith Campos Cordero}

\maketitle
\hrulefill
\begin{abstract}
We consider functionals of the form 
\begin{equation*}
\mathcal{F}(u):=\int_\Omega\!F(x,u,\nabla u)\,\dd x,
\end{equation*}
where $\Omega\subseteq\R^n$ is open and bounded. The integrand $F\colon\Omega\times\R^N\times\R^{N\times n}\to\R$ is assumed to satisfy the classical assumptions of a power $p$-growth and the corresponding strong quasiconvexity. In addition, $F$ is  H\"older continuous with exponent $2\beta\in(0,1)$ in its first two variables uniformly with respect to the third variable, and bounded below by a quasiconvex function depending only on $z\in\R^{N\times n}$. We establish that strong local minimizers of $\mathcal{F}$ are of class $\CC^{1,\beta}$ in an open subset $\Omega_0\subseteq\Omega$ with $\mathcal{L}^n(\Omega\setminus\Omega_0)=0$. This partial regularity also holds for a certain class of weak local minimizers at which the second variation is strongly positive and satisfying a $\mathrm{BMO}$-smallness condition. This extends the partial regularity result for local minimizers by Kristensen and Taheri (2003) to the case where the integrand depends also on $u$. Furthermore, we provide a direct strategy for this result, in contrast to the blow-up argument used for the case of homogeneous integrands. 


\vspace{0.1cm}

\noindent\textsc{Date:} \today{}
\end{abstract}

\section{Introduction}
We investigate the regularity properties of $\WW^{1,q}$-local minimizers of variational integrals of the form
\begin{equation}\label{eq:functional}
\mathcal{F}(u):=\int_\Omega\!F(x,u,\nabla u)\,\dd x.
\end{equation}
Here, $\Omega\subseteq\R^n$ is an open and bounded set and $u\in\WW^{1,p}(\Omega,\R^N)$ for a fixed $p\geq 2$. We recall at this point that, for a given $q\in[1,\infty]$,  a map $u\in\WW^{1,p}(\Omega,\R^N)$ is said to be a $\WW^{1,q}$-\textbf{local minimizer} if there exists a $\delta>0$ such that, for every $\varphi\in\WW^{1,q}_0(\Omega,\R^N)$ satisfying $\|\nabla\varphi\|_{\LL^q(\Omega,\R^{N\times n})}<\delta$, it holds that $\mathcal{F}(u)\leq\mathcal{F}(u+\varphi)$. When $q\in[1,\infty)$, we say that $u$ is a strong local minimizer, whereas if $q=\infty$, we call $u$ a weak local minimizer.

Regarding the integrand $F\colon\Omega\times\R^N\times\R^{N\times n}\to\R$, we assume that it satisfies the following standard conditions:
\begin{enumerate}[start=0,label={(H\arabic*)}]
\item\label{H0} $F=F(x,u,z)$ is a continuous function, of class $\CC^2$ in the variable $z\in\R^{N\times n}$, and such that $F_{zz}$ is continuous in $\Omega\times\R^N\times\R^{N\times n}$.
\item \label{H1:pg} $F$ satisfies a natural growth condition: for a fixed $p\geq 2$ and for a constant $L>0$ it holds that, for every $(x,u,z)\in \Omega\times\R^N\times\R^{N\times n}$,
\begin{equation*}
|F(x,u,z)|\leq L(1+|z|^p).
\end{equation*}
\item \label{H2:QC} $F$ is strongly quasiconvex: there is a constant $\ell>0$ such that, for every $(x_0,u_0,z_0)\in \Omega\times\R^N\times\R^{N\times n}$ and every $\varphi\in\WW^{1,p}_0(\Omega,\R^N)$, 
\begin{equation*}
\ell\int_\Omega\! |V(\nabla\varphi)|^2\,\dd x\leq \int_\Omega\! \left( F(x_0,u_0,z_0+\nabla\varphi)-F(x_0,u_0,z_0) \right) \,\dd x,
\end{equation*}
where the auxiliary function $V\colon\R^{N\times n}\to\R$ is defined by
\[
V(z):=(|z|^2+|z|^p)^{\frac{1}{2}}. 
\]
\item\label{H3:GleqF} There exists a function $G\colon \R^{N\times n}\to\R$ such that, for every $(x,u,z)\in\Omega\times\R^N\times\R^{N\times n}$, 
\begin{equation*}
G(z)\leq F(x,u,z).
\end{equation*}
Furthermore, $G$ is assumed to be strongly quasiconvex, in the sense that $G$ is continuous and, for a constant $\ell_G>0$, it holds that for every $z_0\in\R^{N\times n}$ and every $\varphi\in\WW^{1,p}_0(\Omega,\R^N)$,
\[
\ell_G\int_{\Omega}\!|\nabla\varphi|^p\,\dd x\leq \int_{\Omega} \left(G(z_0+\nabla\varphi)-G(z_0) \right)\,\dd x.
\]
\item\label{H4:HoldCont} There is an increasing function $\rho(s)\geq 2L$ such that, for a constant $c>0$ and for every $(x,u,z),(y,v,z)\in\Omega\times\R^N\times\R^{N\times n}$,
\begin{equation*}
|F(x,u,z)-F(y,v,z)|\leq c\, \vartheta(|v|,|x-y|^2+|u-v|^2)(1+|z|^p),
\end{equation*}
with $\vartheta(s,t)=\min\{2L,\rho(s)t^{\beta}\}$, $\beta\in(0,\frac{1}{2})$.
\end{enumerate}

On the other hand,  when dealing with weak local minimizers, we will replace assumption \ref{H0} by the following regularity condition on $F$:
\begin{enumerate}[start=0,label={(H\arabic*$_\mathrm{w}$)}]
\item\label{H0b:C2}  $F=F(x,u,z)$ is a continuous function, of class $\CC^2$ in the variables $(u,z)\in\R^N\times \R^{N\times n}$, and such that $F_{uu}$ and $F_{zz}$ are continuous in $\Omega\times\R^N\times\R^{N\times n}$.
\end{enumerate}
For the case of weak local minimizers we will also require the following growth condition on $F_u$:
\begin{enumerate}[start=0,label={(H1$_\mathrm{w}$)}]
\item\label{H1b:Fu}  
For a constant $K>0$ it holds that, for every $(x,u,z)\in \Omega\times\R^N\times\R^{N\times n}$,
\begin{equation*}
|F_u(x,u,z)|\leq K(1+|z|^p).
\end{equation*}
\end{enumerate}

\begin{remark}
Assumption \ref{H3:GleqF} is a coercivity condition on $F$ that is uniform in $(x,u)\in\Omega\times\R^N$, see \cite{CYCJK}. In this regard, we emphasize that the value of the constants $\ell$ and $\ell_G$ does not play an important role in our context and, furthermore, we can normalize the integrand $F$  by $\min\{\ell,\ell_G\}$ while the set of local minimizers remains unchanged. Whereby, for the rest of the paper we shall assume, without loss of generality, that $\ell=\ell_G=1$. 
\end{remark}

We note that \ref{H3:GleqF}  plays a fundamental role in obtaining a preliminary higher integrability on the local minimizers under investigation. On the other hand, assumption \ref{H4:HoldCont} is a H\"older continuity condition on the $(x,u)$ variables and uniform in $z$. 

We also record that assumptions \ref{H3:GleqF} and \ref{H4:HoldCont} enable us to establish suitable estimates for the case under study, in which $F$ depends on the lower order terms $x\in\Omega$ and $u\in\R^{N}$, in addition to its dependence on $z\in\R^{N\times n}$. A corresponding assumption of the form of \ref{H4:HoldCont} is also necessary when $F$ depends only on $x$ and $z$, see \cite{JCC2017}, while condition \ref{H3:GleqF} is only required when the dependence on $u$ is present. However, we recall that all the hypotheses \ref{H0}-\ref{H4:HoldCont} are standard when proving partial regularity of global minimizers if the integrand depends also on $(x,u)$, see, for example, \cite{AF87, Beck2011, GiaqMod86, Giusti,Hamburger96, KM07}.

The study of local minimizers in the calculus of variations has been broadly motivated by models arising in materials science. In the same spirit, a fundamental example of John in \cite{John} establishes non-uniqueness of equilibrium points in nonlinear hyperelasticity. The annular shape of the domain in this example and the symmetric properties of the possible solutions to the equilibrium equations provided strong evidence that the topology of $\Omega$  plays a central role in the number of local minimizers admitted by the functional. 

The question of existence of local minimizers was addressed by Post and Sivaloganathan for low dimensions in \cite{PS1997}, see also \cite{Tah2001a}. Kristensen and Taheri further extended the previous examples to obtain $\WW^{1,p}$-local minimizers for homogeneous integrands satisfying a natural $p$-growth condition. Furthermore, it was shown by Taheri in \cite{Tah03} that, if the domain is star-shaped, then \textit{strong} local minimizers are actually global minimizers if they are subject to linear boundary conditions. The problem of existence was finally settled in great generality by Taheri in \cite{Tah05}, where a lower bound for the number of local minimizers of functionals of the form \eqref{eq:functional} is given in terms of certain topological properties of the domain. 

The foundations of regularity theory for solutions to elliptic systems rely on the astounding works of De Giorgi \cite{DG61} and Almgren \cite{Al68,Al76}, that were followed by those of Giusti and Miranda \cite{GiMi}, Morrey \cite{Morrey}, and others. Later on, Giaquinta and Giusti \cite{GiaqGius} established partial regularity for minimizers of functionals as in \eqref{eq:functional} under the assumption of quadratic growth and strong convexity in the $z$ variable. Many other works followed in the case of convexity in the $z$-variable, see \cite{Gia, Giusti}. 

On the other hand, in the quasiconvex setting it wasn't until the work of Evans \cite{Eva86} that partial regularity of class $\CC^{1,\alpha}$ was established for minimizers, under the assumption that the integrand satisfies a controlled quadratic growth. Several important generalizations of his work have been established, including  \cite{AF87,AF89,CFP18,CFM98,CFLM20,DLSV12,DM04,EG87,FH85,FH91,Gia1988,GK19,Hamburger96,CH16,
KM07,TSchmidt}.   Of particular relevance to us is the direct approach to establish partial regularity of minimizers of functionals of the form \eqref{eq:functional}, with $F$ satisfying a $p$-growth as in \ref{H1:pg} and a strong quasiconvexity condition of the form \ref{H2:QC}. The foundations of this method go back to the early works of \cite{Gia}-\cite {GiaqMod86}. 

In the quasiconvex setting, a full regularity result under a smallness condition on the boundary datum has been obtained in \cite{JCCJKreg}. Furthermore, the impact of regularity on uniqueness has been studied by the author and Jan Kristensen in \cite{JCCJKuniq}, where  a uniqueness result for global minimizers of quasiconvex functionals has been obtained, once again under a natural smallness restriction on the Dirichlet boundary condition.  

Regarding the case of local minimizers, Kristensen and Taheri established in \cite{KT} that, under the classical assumptions of natural growth and strong quasiconvexity of a homogeneous integrand $F=F(z)$, $\WW^{1,q}$-local minimizers are of class $\CC^{1,\alpha}$ in a subset of the domain $\Omega$ of full $n$-dimensional measure. They observed that an important challenge when dealing with $\WW^{1,q}$-local minimizers, if $q>p$,  is that of obtaining higher integrability properties. An important reason leading to this obstacle is that, while establishing partial regularity via an indirect blow-up argument, the Caccioppoli inequality that can be obtained carries on the right hand side a term of the form \begin{equation}\label{eq:thetaCacc}
\theta\int_{B(x_0,r)}\! |\nabla u_j-\nabla a_j|^p\,\dd x,
\end{equation} where $\theta\in(0,1)$, $(u_j)$ is the blown-up sequence and $(a_j)$ is a sequence of suitable affine maps. In the case of global minimizers, this term can be made to disappear after an iteration process as in \cite{Eva86}. However, in the case of local minimizers, such an iteration would require letting $j\to\infty$, which cannot be done at this stage since the sequence $(u_j)$ is only known to converge weakly in $\WW^{1,p}$. On the other hand, Evans and Gariepy had already observed in \cite{EG87} that a Caccioppoli inequality of the first kind (in the sense of \cite[Ch. 9]{Giusti}) is not really necessary to establish convergence in $\WW^{1,2}$ of the blown-up sequence. However, their argument does not resolve the issue for strong local minimizers. In \cite{KT}, the obstacle of not being able to iterate the Caccioppoli inequality carrying the term in \eqref{eq:thetaCacc} is overcome for $\WW^{1,q}$-local minimizers by means of a measure-theoretical argument, thanks to which strong convergence can be established bypassing the need of the usual estimates that give higher integrability. Other interesting regularity results concerning local minimizers can be found in \cite{Beck2011,BG17,CN03,CFLM20,DLSV12,ScSc}.

In \cite{JCC2017} it was shown that the blow-up strategy can be suitably adapted to obtain partial regularity of minimizers when the integrand is of the form $F(x,z)$, with $(x,z)\in\Omega\times\R^{N\times n}$ and the treatment can even be taken to establish partial regularity up to the boundary of $\Omega$. However, the case in which the integrand $F$ depends also on $u\in\R^N$ poses extra technical challenges. Such difficulties arise mainly from the fact that a uniform coercivity assumption of the form \ref{H3:GleqF} cannot suitably be \textit{blown up} and, just as in the case of global minimizers for functionals with lower order terms, see \cite{GiaqMod86}, the strong quasiconvexity condition \ref{H2:QC} also fails to be enough to establish a pre-Caccioppoli inequality in this case. 

On the other hand, and in remarkable contrast with the aforementioned regularity results,   M\"uller and {\v{S}}ver{\'a}k \cite{MS03} built examples of Lipschitz solutions to the weak Euler-Lagrange equations associated to a quasiconvex variational problem, but that are nowhere of class $\CC^1$.  Kristensen and Taheri \cite{KT} provided a modification of these examples so that the second variation at the Lipschitz solutions to the Euler-Lagrange equations was actually uniformly positive, and hence the solutions were, in fact,  weak local minimizers. Furthermore, the importance of the regularity properties of strong local minimizers has also been made evident in the sufficiency result established by Grabovsky and Mengesha in \cite{GM}, see also \cite{JCCKK}. 

In this setting, the main objective of this work is to establish a partial $\CC^{1,\alpha}$-regularity result for $\WW^{1,q}$-local minimizers of variational functionals of the form \eqref{eq:functional}, where the integrand also depends on $u$.  It is worth noting that, just as in the case of homogenous integrands first established in \cite{KT}, and then extended to integrands with $x$ dependence in \cite{JCC2017}, we require the local minimizer to be a priori a $\WW^{1,q}_{\mathrm{loc}}$ map. It remains unclear whether this assumption is really necessary. 

As usual, our regularity result will rely on a suitable decay of the mean oscillations denoted by
\[
E(y,r):=\dashint_{B(y,r)} \!|V(\nabla u - (\nabla u)_{y,r})|^2\,\dd x,
\] 
were $B(y,r)\subseteq\Omega$ is a ball and the bar on the integral signifies its mean value, as specified in the list of notations below. 

Our main theorem is the following. 

\begin{theorem}\label{theo:regularity}
Let $F\colon\Omega\times\R^N\times \R^{N\times n}\to\R$ be an integrand of class $\CC^2$ in the variable in $\R^{N\times n}$ and satisfying {\ref{H0}-\ref{H4:HoldCont}}. Let $q\in [2,\infty]$ and assume that $u\in\WW^{1,p}(\Omega,\R^N)\cap \WW^{1,q}_\mathrm{loc}(\Omega,\R^N)$ is a $\WW^{1,q}$-local minimizer of $\mathcal{F}$. 
Let 
\[
\Sigma_0:=\left\{  x\in\Omega \,\,\mathrm{ : }\,\, \limsup_{r\to 0}|(u)_{x,r}|+|(\nabla u)_{x,r}|=\infty\,\,\,\mathrm{ or }\,\,\liminf_{r\to 0} E(x,r)>0   \right\}
\]
and define $\Omega_0:=\Omega\setminus\Sigma_0$. Then, $\Omega_0$ is open, $\mathcal{L}^n(\Sigma_0)=0$ and, if $q\in[2,\infty)$, then $u$ is locally of class $\CC^{1,\beta}$ in $\Omega_0$. On the other hand, if $q=\infty$, we assume in addition that the second variation at $u$ is strongly positive, as in \eqref{H5:SecVar}, and that assumptions {\ref{H0b:C2}} and {\ref{H1b:Fu}} also hold. Then, there exists a $\delta_0>0$ such that, if $\delta_1\in(0,\delta_0)$ and if
\begin{equation}\label{eq:q=infty}
\limsup_{r\to 0} \|\nabla u (y) - (\nabla u)_{x,r} \|_{\LL^\infty(B(x,r),\R^{N\times n})} < \delta_1 
\end{equation}
locally uniformly in $x\in\Omega$, then $u$ is also locally of class $\CC^{1,\beta}$ in $\Omega_0$.
\end{theorem}
Note that, by the examples cited above, condition \eqref{eq:q=infty} is indeed necessary for the case of $\WW^{1,\infty}$-local minimizers. 

The strategy that we follow is fundamentally inspired by the direct argument of Giaquinta and Modica \cite{GiaqMod86} to establish partial regularity of global minimizers. The application of this approach is new in the context of $\WW^{1,q}$-local minimizers. An important challenge to overcome is that, just as for the case without $u$-dependence, we can only obtain a Caccioppoli of the first kind, which is stated in Theorem \ref{theo:Caccioppoli}. 

As in the homogeneous case, the restrictions imposed by the fact that we are dealing with local minimizers do not allow us to perform an iteration to force the term  ${\theta}\, \dashint_{B_r }  |V( \nabla u - z_0) |^2  \, \dd  x $ to vanish. However, a generalized version of Gehring's Lemma, originally established in \cite{Geh73}, was obtained by Stredulinsky in \cite{Stredulinsky}. This generalized version enables the Caccioppoli inequality  to self-improve into a similar one with higher exponents on all the relevant terms. Whereby, by combining this with a preliminary higher integrability obtained in a classical way, we can finally obtain a suitable reverse H\"older inequality, which is stated in Theorem \ref{theo:ReverseHolder}. 

The reverse H\"older inequality will finally allow us to establish a good decay rate for the mean oscillations of $\nabla u$. This is achieved by comparing the mean oscillations of the minimizer $u$ with those of the (regular) minimizer of a second order Taylor polynomial of a frozen integrand of the form $F(x_0,u_0,z)$, as we show in Theorem \ref{theo:linearization}.

The main body of the paper is organized as follows. In Section \ref{Sec:preliminaries}  we recall classical estimates for the integrands satisfying our main assumptions. Furthermore, we state Stredulinsky's version of Gehring's Lemma, being one of the main ingredients of our regularity proof. The final part of Section \ref{Sec:preliminaries} includes the statements of the decay rate and the $\LL^p$-estimates satisfied by $\mathbb{A}$-harmonic maps, namely, the minimizers of strongly quasiconvex quadratic functionals. The main content of this work is contained in Sections \ref{Sec:WeakLocMin} and \ref{section:regularity}. In Section \ref{Sec:WeakLocMin} we prove that weak local minimizers at which the second variation is strongly positive are, in fact, $\WW^{1,\mathrm{BMO}}$-local minimizers, in the sense that $u$ minimizes the energy under perturbations for which the $\mathrm{BMO}$-seminorm of the derivative is small.  Finally, in Section \ref{section:regularity} we establish the proof  of Theorem \ref{theo:regularity}. Concerning the different steps of the proof, we remark that the linearization process in Subsection \ref{Sec:Linearization} makes use of the aforementioned $\WW^{1,\mathrm{BMO}}$-minimality result  for the case of weak local minimizers, since it is by these means that we can use the local minimality property when comparing the energy at $u$ with that of the solution to the linearized problem. 

We now introduce the notation that we use in the rest of the paper. 
\subsection{Notation:} Throughout this work, we shall use the following set of notational conventions:
\begin{itemize}
\item $|\cdot|$ denotes the Euclidean norm in any space $\R^m$. If $m=N\times n$, so that we are in a space of matrices, it denotes the trace norm, so that $|z|:=(\mathrm{Tr}(zz^t))^{\frac{1}{2}}$; 
\item $B(x_0,r)$ denotes the open ball centred at $x_0\in \R^n$ and with radius $r>0$;
\item $\overline{B(x_0,r)}$ denotes the closure in $\R^n$ of $B(x_0,r)$; 
\item if we have fixed a ball $B(y,r)$, we shall often write $B_r$ instead of $B(y,r)$;
\item if $\zeta>0$ and we have fixed a ball $B(y,r)$, we shall denote by $B_{\zeta r}$ the ball still centred at $y$ and with radius $\zeta r$. 
\item for a given open set $\omega\subseteq\R^n$ of finite Lebesgue measure $\mathcal{L}^n(\omega)\in(0,\infty)$, and for a given map $f\colon\omega\to\R^m$, we denote the average value of $f$ over $\omega$ by
\[
(f)_\omega=\frac{1}{\mathcal{L}^n(\omega)}\int_\omega \! f\,\dd x=\dashint _\omega \! f\,\dd x.
\]
In the particular case that $\omega=B(x_0,r)$, we denote $(f)_{x_0,r}:=(f)_\omega$;
\item for a function $f\in\CC^{1,\gamma}(\overline{\omega},\R^m)$, with $\gamma\in(0,1)$, $\|f\|_{\CC^1}:=\|f\|_{\CC^1(\overline{\omega},\R^m)} + [\nabla f]_{\CC^{0,\alpha}} $, where $ [\nabla f]_{\CC^{0,\alpha}} $ denotes de $\alpha$-H\"older semin-norm of $\nabla f$;
\item for a function $f\in\WW^{1,s}(\Omega,\R^N)$, we denote $\|f\|_{1,s}:=\|f \|_{\LL^s(\Omega,\R^N)}+\|\nabla f \|_{\LL^s(\Omega,\R^{N\times n})}$;
\item for $1\leq p<\infty$, $\WW^{1,p}_0(\Omega,\R^N)$ refers to the closure in $\WW^{1,p}$ of the space of compactly supported smooth functions $\CC^{\infty}_c(\Omega,\R^N)$.
\\ On the other hand, $\WW^{1,\infty}_0(\Omega,\R^N):= \WW^{1,\infty}(\Omega,\R^N)\cap \WW^{1,1}_0(\Omega,\R^N)$. 
\\
For $1\leq p\leq \infty$ and $u\in\WW^{1,p}(\Omega,\R^N)$, $\WW^{1,p}_u(\Omega,\R^N)$ denotes the affine space $u+\WW^{1,p}_0(\Omega,\R^N)$;
\item the letter $c$ will be used to denote a constant and it may change its specific value from line to line in a chain of equations. 
\end{itemize}

\section{Preliminary results}\label{Sec:preliminaries}
It is well known  that, under the assumptions \ref{H1:pg} and \ref{H2:QC}, the following estimates hold: there exists a constant $c=c(L)>0$ such that, $\forall(x,u,z),(x,u,w)\in \Omega\times\R^N\times\R^{N\times n}$, 
\begin{equation}\label{LipFz}
|F(x,u,z)-F(x,u,w)|\leq c (1+|z|^{p-1}+|w|^{p-1})|z-w|
\end{equation}
and

\begin{equation}
|F_z(x,u,z)|\leq c(1+|z|^{p-1}).
\end{equation}
See \cite{Dac,Giusti} as references for these estimates. 

Another classical and useful result concerns the following growth properties for shifted frozen integrands:
\begin{lemma}\label{lemma:FrozenShift}
Let $m>0$ and let $(x_0,u_0,z_0)\in\R^{N\times n}$ be fixed, with $|u_0|+|z_0|\leq m$. Assume that $F\colon\Omega\times\R^N\times\R^{N\times n}\to\R$ is of class $\CC^2$ in the $z$ variable and that it satisfies assumptions \ref{H1:pg} and \ref{H2:QC}. We define the shifted frozen integrand 
\begin{align*}
\bar{F}(z):=&F(x_0,u_0,z_0+z)-F(x_0,u_0,z_0)-F_z(x_0,u_0,z_0)[z]\\
& = \int_0^1\! (1-t)F_{zz}(x_0,u_0,z_0+tz)[z,z]\,\dd x.
\end{align*}
Then, there exists a constant $c=c(L,m,F'')>0$ such that,  for every $z\in\R^{N\times n}$, we have that
\begin{align*}
|\bar{F}(z)|& \leq c\left(|z|^2+|z|^p \right)\\
|\bar{F}'(z)|& \leq c\left(|z|+|z|^{p-1}\right).
\end{align*}
Furthermore, for every $z,w\in\R^{N\times n}$, it holds that
\begin{align*}
|\bar{F}(z)-\bar{F}(w)|\leq c \left(|z| +|w| +|z|^{p-1} + |w|^{p-1}  \right)|z-w|.
\end{align*}
\end{lemma}
Under the assumption that $|u_0|+|z_0|\leq m$, the proof of this lemma follows from the result of Acerbi \& Fusco in \cite[Lemma 2.3]{AF87}.

We now state the following version of Gehring's Lemma. A proof of this result can be found in \cite[Ch. V. Proposition 1.1]{Gia}.  The following version of the result was proved by Stredulinsky in \cite{Stredulinsky}. See \cite{ME75} and \cite[Proposition 5.1]{GiaqMod} for earlier developments generalizing Gehring's Lemma in this direction.

\begin{theorem}\label{theo:Gehring}
Let  $\mathbb{B}\subseteq\R^n$ be an open ball and let $p>1$. For $M\geq 1$, let $f,g\in \LL^p(\mathbb{B},\R^M)$, $\zeta\in(0,1)$ and assume that $\theta\in (0,1)$ and $K>0$ are constants such that, for every $B_r=B(x_0,r)\subseteq \mathbb{B}$, 
\[
\dashint_{B_{\zeta r}} |g|^p \,\dd x\leq \theta\, \dashint_{B_{ r}} |g|^p \,\dd x + K\left( \dashint_{B_{ r}} |g| \,\dd x \right)^p + \dashint_{B_{ r}} |f|^p \,\dd x.
\]
Then, there exist $\varepsilon=\varepsilon(K,\theta,q,n)>0$ and a constant $c=c(K,\theta,q,n)>0$ such that, for every $q\in [p,p+\varepsilon)$, it holds that $g\in \LL^q(\frac{1}{2}\mathbb{B},\R^M)$ and
\[
\left(\dashint_{\frac{1}{2}\mathbb{B}} |g|^q \,\dd x\right)^{\frac{1}{q}}\leq c\, \left(\dashint_{\mathbb{B}} |g|^p \,\dd x\right)^{\frac{1}{p}} + c\,\left( \dashint_{\mathbb{B}} |f|^q \,\dd x\right)^{\frac{1}{q}}.
\]

\end{theorem}

The following estimate is of standard use in regularity theory and it is an easy consequence of the convexity of the function $V$ and Jensen's inequality. We shall make use of it freely for different domains $\omega\subseteq\R^n$. 
\begin{lemma}\label{lemma:qminofMean}
Let $p\geq 2$ and let $\omega\subseteq\R^n$ be an open set. Then, there exists a constant $c=c(p)>0$ such that, for $n,M\geq 1$,  every $f\in \LL^p(\omega,\R^M)$ and every $\xi\in\R^{M}$, 
\begin{equation}
\dashint_\omega\!|V(f-(f)_\omega)|^2\, \dd x\leq c\, \dashint_\omega\!|V(f-\xi)|^2\, \dd x.
\end{equation}
\end{lemma}
We conclude this section by stating the following two results for $\mathbb{A}$-harmonic maps. The regularity properties satisfied by these is the fundamental cornerstone that enables us to establish partial regularity in the general nonlinear case. 
\begin{theorem}\label{theo:DecayAHarmonic}
Let $\Omega\subset\R^n$ be open and bounded. Assume that $\mathbb{A}\colon\R^{N\times n}\times\R^{N\times n}\to\R$ is a symmetric bilinear form satisfying that
\begin{itemize}
\item[(i)] for every $\xi,\eta\in\R^{N\times n}$, $\mathbb{A}[\xi,\eta]\leq L|\xi||\eta|$ and
\item[(ii)] $\mathbb{A}$ is strongly quasiconvex, i.e., for every $\varphi\in\WW^{1,2}_0(\Omega,\R^N)$, 
\begin{equation}\label{QCA}
2\int_\Omega\! |\nabla\varphi|^2 \,\dd x\leq \int_\Omega\! \mathbb{A}[\nabla\varphi,\nabla\varphi]\,\dd x. 
\end{equation}
\end{itemize} 
Let $u\in\WW^{1,p}(\Omega,\R^N$ and take $B_R\subseteq\Omega$ to be an open ball. Then, there exists a unique $h\in\WW^{1,2}_u(B_R,\R^N)$ minimizing the functional
\[
v\mapsto \int_{B_R}\! \mathbb{A}[\nabla v,\nabla v]\,\dd x,
\]
over $\WW^{1,2}_u(B_R,\R^N)$ and, for a constant $c=c(n,L)$  it holds that, for every $r\in(0,R)$, 
\[
\dashint_{B_r}\! |V(\nabla h - (\nabla h)_r)|^2 \,\dd x \leq \,c\,\left( \frac{r}{R}\right)^2\dashint_{B_R}\! |V(\nabla h - (\nabla h)_R)|^2 \,\dd x.
\]
\end{theorem}
Good references for this classical result are \cite{GiaqMar}, \cite[Theorem 10.7]{Giusti}. 

It is worth recalling at this point that if a function $F^0\colon\R^{N\times n}\to\R$ is of class $\CC^2$ and satisfies the strong quasiconvexity condition that for every $z_0\in\R^{N\times n}$ and every $\varphi\in\WW^{1,2}_0(\Omega,\R^N)$ it holds that
\[
\int_\Omega\! |V(\nabla\varphi)|^2\,\dd x\leq \int_\Omega\! \left( F^0(z_0+\nabla\varphi)-F^0(z_0) \right) \,\dd x,
\]
then for every $\varphi\in \WW^{1,2}_0(\Omega,\R^N)$ the following strong quasiconvexity condition is satisfied:
\begin{equation}\label{eq:F''qc}
2 \int_\Omega\! |\nabla\varphi|^2\,\dd x\leq \int_\Omega\!  F^0_{zz}(z_0)[\nabla\varphi,\nabla\varphi]  \,\dd x.
\end{equation}
Indeed, for $\varphi\in \CC^{\infty}_0(\Omega,\R^N)$ this follows from the fact that $t=0$ minimizes the real valued function 
\[
\mathcal{J}(t):= \int_\Omega\! \left( F^0(z_0+t\nabla\varphi)-F^0(z_0)- |V(t\nabla\varphi)|^2 \right) \,\dd x
\]
and, hence, 
\begin{equation}
0\leq \mathcal{J}''(0)= \int_\Omega\! \left( F^0_{zz}(z_0)[\nabla\varphi,\nabla\varphi] -2|\nabla\varphi|^2 \right) \,\dd x.
\end{equation}
For $\varphi\in\WW^{1,2}_0(\Omega,\R^N)$, the inequality \eqref{eq:F''qc} follows then by approximation. 
\\
\\
The following result concerns the well known $\LL^p$-estimates for $\mathbb{A}$-harmonic maps. 
\begin{theorem}\label{Lpestimates}
Let $\mathbb{A}\colon\R^{N\times n}\times\R^{N\times n}\to\R$ be a symmetric bilinear form as in Theorem \ref{theo:DecayAHarmonic} and let $u\in \WW^{1,q}(B_R,\R^N)$. If $h\in\WW^{1,2}_u(B_R,\R^N)$ minimizes the functional
\[
v\mapsto \int_{B_R}\! \mathbb{A}[\nabla v,\nabla v]\,\dd x,
\]
and if $q\in [2,\infty)$, there exists a constant $K_q=K_q(L,n,q)>0$ such that, for every constant vector $z_0\in\R^{N\times n}$, 
\[
\|\nabla h-z_0\|_{\LL^{q}(B_R,\R^{N\times n})}\leq K_q \|\nabla u-z_0\|_{\LL^{q}(B_R,\R^{N\times n})}.
\]
On the other hand, if $u\in\WW^{1,\infty}(\Omega,\R^N)$,  there exists a constant $A_\infty=A_\infty(L,n)>0$ such that, for every constant vector $z_0\in\R^{N\times n}$, 
\[
[\nabla h-z_0]_{\mathrm{BMO}(B_R,\R^{N\times n})}\leq A_\infty \|\nabla u-z_0\|_{\LL^{\infty}(B_R,\R^{N\times n})}.
\]
\end{theorem}
 The proof for $q=2$ follows from the strong quasiconvexity condition. The proof for $q\in (2,\infty]$ can be found in \cite[Theorem 2.14 and Theorem 10.15]{Giusti}.

\section{Weak local minimizers}\label{Sec:WeakLocMin}

\subsection{An improved local minimality property with $\mathrm{BMO}$ variations}

In this section we will establish that weak local minimizers at which the second variation is strictly positive are, in fact, $\WW^{1,\mathrm{BMO}}$-local minimizers, in the terminology of \cite{JCC2017}. This was first observed for homogeneous integrands in \cite[Theorem 6.1]{KT} (see also \cite[Theorem 4.4]{JCC2017}). The impact of small $\WW^{1,\mathrm{BMO}}$-variations has been studied in \cite{SpectorSpectorBMO}, while interesting results concerning the regularity of extremals with $\mathrm{BMO}$-small gradient recently appeared in  \cite{CIrving}. 

\begin{definition}
Let $\phi\in \LL^1(\Omega,\R^{N\times n})$. We say that $\phi$ is of \textbf{bounded mean oscillation} if and only if 
\begin{equation*}
\underset{B(x,r)\subseteq\Omega}{\sup}\underset{B(x,r)}{\dashint}|\phi-(\phi)_{x,r}|\,\dd y<\infty.
\end{equation*}
In this case, we define the semi-norm
\begin{equation*}
[\phi]_{\mathrm{BMO}(\Omega,\R^{N\times n})}:=\underset{B(x,r)\subseteq\Omega}{\sup}\underset{B(x,r)}{\dashint}|\phi-(\phi)_{x,r}|\,\dd y<\infty
\end{equation*}
and we set
\begin{equation*}
\mathrm{BMO}(\Omega,\R^{N\times n}):=\bigl\{\phi\in \LL^1(\Omega,\R^{N\times n})\,\,\mathrm{ : } \,\,[\phi]_{\mathrm{BMO}(\Omega,\R^{N\times n})}
<\infty\bigr\}.
\end{equation*}
\end{definition}
In order to simplify the notation later on, for $(x,v,w)\in \Omega\times\R^n\times \R^{N\times n}$ we define the bilinear form $L(x,u,z)$ by
\begin{align}
L(x,u,z)[(v,w),(\hat{v},\hat{w})]:= &F_{uu}(x,u,z)v\cdot \hat v+ F_{uz}(x,u,z)v\cdot w+F_{uz}(x,u,z)\hat{v}\cdot \hat{w} \notag
\\
&+F_{zz}(x,u,z)w\cdot \hat{w}\label{eq:L}
\end{align}
for every $v,\hat{v}\in\R^N$ and every $w,\hat{w}\in\R^{N\times n}$. 

We shall also denote $F(x):=F(x,u(x),\nabla u(x))$ and we adopt the corresponding notation for all the partial derivatives of $F$. 

For the last part of this section we will assume that $u\in\WW^{1,\infty}(\Omega,\R^N)$ is a weak local minimizer of $\mathcal{F}$ at which the second variation is uniformly strictly positive, meaning that there exists a constant $c_0>0$ such that, for all $\varphi\in\WW_0^{1,\infty}(\Omega,\R^N)$, 
\begin{align}
&\int_{\Omega}\left(F_{uu}(x)[\varphi(x),\varphi(x)] + 2F_{uz}(x)[\nabla\varphi(x),\varphi(x)] + F_{zz}(x)[\nabla\varphi(x),\nabla\varphi(x)]\right)\dd x\notag \\
=&\int_{\Omega}L(x,u(x),\nabla u(x))[(\varphi,\nabla\varphi),(\varphi,\nabla\varphi)]\,\dd x\label{H5:SecVar}\tag{H5}\\
\geq & c_0 ||\nabla\varphi ||^2_2.\notag
\end{align}

\begin{lemma} \label{LemmagrowthG}
Let $F\colon{\Omega}\times \R^N\times \R^{N\times n}\rightarrow\R$ be such that  {\ref{H0b:C2}}, \ref{H1:pg}, \ref{H1b:Fu} and \ref{H2:QC} hold for some $p\geq 2$. Let $u\in\WW^{1,\infty}(\Omega,\R^N)$ and define the function $G\colon{\Omega}\times \R^N\times \R^{N\times n}\rightarrow\R$  by 
\begin{align*}
G(x,y,z):=&F(x,u(x)+y,\nabla u(x)+z)-F(x) -F_y(x)[ y]-F_z(x)[z]\notag\\
=&\int_0^1 (1-t)L(x, u(x)+ty,\nabla u(x)+tz)[(y,z),(y,z)],
\end{align*}
where, for each $(x,y,z)\in\Omega\times\R^N\times\R^{N\times n}$, the bilinear form $L(x,u,z)$ is defined as in \eqref{eq:L}.  

Then, there exists a constant $c=c(\|u\|_{{1,\infty}})>0$ such that, for each $x\in{\Omega}$, $y,\hat y\in\R^N$ and $z,\hat z\in\R^{N\times d}$, 
\[
|G(x,y,z) - G(x,\hat y,\hat z)|\leq c\left(A_{p-1}(y,z,\hat y,\hat z)|z-\hat z| + A_p(y,z,\hat y,\hat z)|y-\hat y|\right),
\]
where
\[
A_p(y,z,\hat y,\hat z) = |y|+|\hat y|+|z|+|\hat z|+|z|^p+|\hat z|^p.
\]
\end{lemma}
The proof of this result can be found in \cite[Lemma 4.4]{JCCKK} (see also \cite{GM}). It follows the truncation strategy developed by Acerbi \& Fusco in \cite{AF87}.

The regularity of the second partial derivatives of $F$ can be phrased in terms of the existence of a modulus of continuity for the function $L$, as we settle in the following lemma. 
\begin{lemma}\label{lemma:modL}
Assume that $F\colon\Omega\times\R^N\times\R^{N\times n}\to\R$ satisfies the condition \ref{H0b:C2} and let $m>1$ be fixed. There exists a modulus of continuity $\omega_L\colon[0,\infty)\to[0,1]$ such that $\omega_L$ is increasing, concave, $\omega_L(t)\geq 1$ for every $t\geq 1$, $\lim_{t\to 0}\omega_L(t)=0$, and with the property that, for a constant $c=c(m)>0$ and for every $(x,u,z),(x,v,w)\in\Omega\times\R^N\times\R^{N\times n}$ satisfying that $|u|+|v|+|z|+|w|\leq m+1$, it holds that
\begin{equation*}
|L(x,u,z)-L(x,v,w)|\leq c\,\omega_L(|u-v|+|z-w|). 
\end{equation*}
\end{lemma}
The proof follows a standard scheme and we only sketch it here for the convenience of the reader:
\begin{proof}
By assumption the function $L$ is continuous in its domain and it is bounded in any set of the form $\Omega\times \overline{B(0,m+1)}\subseteq\Omega\times\R^N\times\R^{N\times n}$. 
Let
\[
K:=1+\sup_{|u|+|z|\leq m+1}|L(x,u,z)|.
\]
We can then ensure that $K$ is well defined. Define now
\[
\omega_0(t):=\frac{1}{2K}\sup \biggl\{|L(x,u-v,z-w)| \mbox{ \textbf{:} } |u-v|+|z-w|\leq t \mbox{ and }|u|+|v|+|z|+|w|\leq m+1 \biggr\}.
\]
$\omega_L\colon[0,\infty)\to[0,1]$ can then be given as the concave envelope of the function $$\bar{\omega}(t):=\max\{\omega_0(t),\min\{t,1\} \}.$$ The function $\omega_L$ is a modulus of continuity for $L$ in the set ${\Omega}\times \overline{B(0,m+1)} \subseteq {\Omega}\times \R^N\times \R^{N\times n}$  and it satisfies all the desired properties.
\end{proof}

We will require the following definition and the subsequent lemmata, which generalize the Hardy-Littlewood-Fefferman-Stein maximal inequalities. 
\begin{definition}
Let $f\colon\R^n\rightarrow\R^{N\times n}$ be an integrable map. We define the {Hardy-Littlewood maximal function} by
\begin{equation*}
f^\star(x):=\underset{B(y,r)\ni x}{\sup}\,\underset{B(y,r)}{\dashint}|f(y)|\,\dd y,
\end{equation*}
where the supremum is taken over all balls $B(y,r)\subseteq\R^n$ containing $x$. Similarly, the {Fefferman-Stein maximal function} is given by
\begin{equation*}
f^{\#}(x):=\underset{B(y,r)\ni x}{\sup}\,\underset{B(y,r)}{\dashint}|f(y)-(f)_{y,r}|\,\dd y.
\end{equation*}
\end{definition}

\begin{lemma}\label{LemmaHLMaxVMO}
Let $\Phi\colon[0,\infty)\rightarrow[0,\infty)$ be a continuously increasing function with $\Phi(0)=0$. Assume, in addition, that $\Phi(t)=t^p A(t)$ for some $p>1$ and some increasing function $A\colon[0,\infty)\rightarrow[0,\infty)$.  Then, there exists a constant $\gamma=\gamma(n,p)$ such that
\begin{equation}\label{eq:f-f*}
\underset{\R^n}{\int}\Phi(|f|)\dd x\leq \underset{\R^n}{\int}\Phi(f^{\star})\dd x \leq\gamma\underset{\R^n}{\int}\Phi(2|f|)\,\dd x
\end{equation}
for all $f\in \LL^1(\R^n,\R^{N\times n})$.
\end{lemma}
The proof of the first inequality in this lemma follows from the fact that $\Phi$ is increasing and from Lebesgue's Differentiation Theorem, which implies that $|f(x)|\leq f^\star(x)$ for almost every $x\in\R^n$. For a proof of the second inequality we refer the reader to \cite[Lemma 5.1]{GrIwaMos}.
It is well known that both notions of maximal functions are related in the following way.
\begin{lemma}\label{lemmaFSMaxVMO}
Let $\Phi\colon[0,\infty)\rightarrow[0,\infty)$ be a continuously increasing function with $\Phi(0)=0$. Let $\varepsilon>0$ and $f\in  \LL^1(\R^n,\R^{N\times n})$. Then, 
\begin{equation}
\label{eqlemmaFSMaxVMO}
\underset{\R^n}{\int}\Phi(f^\star)\,\dd x\leq \frac{5^n}{\varepsilon}\underset{\R^n}{\int}\Phi\left(\frac{f^\#}{\varepsilon}\right)\,\dd x+2\cdot 5^{3n}\varepsilon\underset{\R^n}{\int}\Phi(5^n2^{n+1}f^\star)\,\dd x. 
\end{equation}
If, in addition, we have that
\begin{equation*}
\underset{t>0}{\sup}\frac{\Phi(2t)}{\Phi(t)}<\infty,
\end{equation*}
we can further conclude that there is a constant $\gamma_1=\gamma_1(n)$ such that, for every $f\in \LL^1(\R^n,\R^{N\times n})$ satisfying that $\underset{\R^n}{\int}\Phi(f^\star)\dd x<\infty$, it holds that
\begin{equation}
\label{eqlemmaFSMax2VMO}
\underset{\R^n}{\int}\Phi(f^\star)\dd x\leq \gamma_1\underset{\R^n}{\int}\Phi(f^\#)\, \dd x.
\end{equation}
\end{lemma}
The proof of (\ref{eqlemmaFSMaxVMO}) can be found, for example, in \cite{KT}. Inequality (\ref{eqlemmaFSMax2VMO}) follows easily from (\ref{eqlemmaFSMaxVMO}) under the given extra assumptions.

An important remark to make at this point is the following:
\begin{remark}\label{remark:PoincPhi}
If $\varphi\in\WW^{1,1}(\Omega,\R^N)$ is a Sobolev map, then for every $x\in\Omega$ we have, by Poincar\'e inequality, that for a constant $c=c(n,N)>0$, 
\begin{equation}
\varphi^\#(x):=\sup_{Q\ni x}\dashint_Q |\varphi-\varphi_Q|\, \dd y\leq c \sup_{Q\ni x}\dashint_Q |\nabla \varphi|\, \dd y  = c (\nabla\varphi)^* (x). 
\end{equation}

\end{remark}The main result of this section can now established and we state it as follows.
\begin{theorem}\label{theo:morethanWLMVMO}
Let $F\colon{\Omega}\times\R^N\times \R^{N\times n}\rightarrow\R$ be a function satisfying \ref{H0b:C2}, \ref{H1:pg}, \ref{H1b:Fu} and \ref{H2:QC} for some $2\leq p<\infty$. Let ${u}\in \WW^{1,\infty}(\Omega,\R^N)$ be an extremal with strictly positive second variation. In other words, assume that
\begin{equation}
\label{uextremalVMO}
\underset{\Omega}{\int} \left( F_y(x,u,\nabla u)[\varphi] + F_z(x,u,\nabla {u})[\nabla\varphi] \right) \dd x=0
\end{equation}
and that \eqref{H5:SecVar} is satisfied. 
Then, there exists a $\delta_*>0$ such that
\begin{equation*}
 \int_\Omega \! \left( F(x,u+\varphi,\nabla u+\nabla\varphi)-F(x,u,\nabla u) \right) \,\dd x 
\geq  \, c\int_{\Omega}\! |\nabla\varphi|^2\,\dd x 
\end{equation*}
for every $\varphi\in \WW_0^{1,\infty}(\Omega,\R^N)$ satisfying $[\nabla\varphi]_{\mathrm{BMO}}\leq\delta_*$.
\end{theorem}
\begin{proof}
Let $\varphi\in \WW_0^{1,\infty}(\Omega,\R^N)$. We define the sets:
\begin{align*}
A:=& \{x\in\Omega\mbox{ : } |\varphi(x)|+|\nabla\varphi(x)|\leq 1 \}; \\
B:= & \{x\in\Omega\mbox{ : } |\varphi(x)|+|\nabla\varphi(x)|>1 \}.
\end{align*}
Then, we have that
\begin{align*}
& \int_\Omega \! \left( F(x,u+\varphi,\nabla u+\nabla\varphi)-F(x,u,\nabla u) \right) \,\dd x\\
= & \int_\Omega \! \left( F(x,u+\varphi,\nabla u+\nabla\varphi)-F(x,u,\nabla u) -F_y(x,u,\nabla u)[\varphi]-F_z(x,u,\nabla u)[\nabla\varphi] \right) \,\dd x \\
= & \left[ \int_\Omega \! \left( F(x,u+\varphi,\nabla u+\nabla\varphi)-F(x,u,\nabla u) -F_y(x,u,\nabla u)[\varphi]-F_z(x,u,\nabla u)[\nabla\varphi] \right)\mathbbm{1}_{B } \,\dd x \right. \\
& \left. - \frac{1}{2}\int_\Omega\!  L(x,u,\nabla u)[(\varphi,\nabla\varphi)(\varphi,\nabla\varphi)] \mathbbm{1}_{B}\,\dd x \right]
\\
& + \left[ \int_\Omega\! \int_0^1 (1-t) L(x,u+t\varphi,\nabla u+t\nabla\varphi)[(\varphi,\nabla\varphi)(\varphi,\nabla\varphi)]\mathbbm{1}_{A} \,\dd t\,\dd x \right. \\
& -  \left. \int_\Omega\int_0^1\! (1-t) L(x,u,\nabla u)[(\varphi,\nabla\varphi)(\varphi,\nabla\varphi)] \mathbbm{1}_{A}\,\dd t\,\dd x\right]\\
& +\frac{1}{2}\int_\Omega\!  L(x,u,\nabla u)[(\varphi,\nabla\varphi)(\varphi,\nabla\varphi)]\,\dd x\\
=: & \,\mathrm{I}+\mathrm{II}+\mathrm{III}. 
\end{align*}
Note first that, by assumption \eqref{H5:SecVar}, 
\begin{equation*}
\mathrm{III}\geq c_0\int_\Omega\! |\nabla\varphi|^2\,\dd x. 
\end{equation*}
Now let $\omega_L\colon[0,\infty)\to[0,\infty)$ be a modulus of continuity for $L$ on the set
\[
\left\{ (x,y,z)\in\Omega\times \R^N \times\R^{N\times n}\mbox{ : } |y|+|z|\leq \|\nabla u\|_\infty+1\right\}
\]
and such that it satisfies the properties given by Lemma \ref{lemma:modL}. Using in particular that $\omega_L$ is increasing, we can then estimate the term $\mathrm{II}$ as follows:
\begin{align*}
\mathrm{II}\geq & - c\int_\Omega\! \omega_L\left(|\varphi|+|\nabla\varphi|  \right)(|\varphi|^2+|\nabla\varphi|^2)\mathbbm{1}_{A}\,\dd x\\
 \geq & - c\int_\Omega\! \omega_L\left(|\varphi|+|\nabla\varphi|  \right)(|\varphi|^2+|\nabla\varphi|^2+|\varphi|^{2p}+|\nabla\varphi|^{2p})\mathbbm{1}_{A}\,\dd x.
\end{align*}
The need to consider the second inequality will become evident once we estimate the remaining term. Indeed, using Lemma \ref{LemmagrowthG} and the fact that $|\varphi(x)|+|\nabla\varphi(x)|>1$ for $x\in B$, we obtain, for a constant $c>0$ that depends on $\|u\|_{{1,\infty}}$, that
\begin{align*}
I\geq & - c\int_\Omega\! (1+|\varphi|+|\nabla\varphi|+|\nabla\varphi|^{p-1})|\nabla\varphi|\mathbbm{1}_{B }\,\dd x \\
& - c\int_\Omega\! (1+|\varphi|+|\nabla\varphi|+|\nabla\varphi|^{p})|\varphi|\mathbbm{1}_{B }\,\dd x \\
\geq & - c\int_\Omega\! (|\varphi|+|\nabla\varphi|+|\nabla\varphi|^{p-1})|\nabla\varphi|\mathbbm{1}_{B } \,\dd x\\
& - c\int_\Omega\! (|\varphi|+|\nabla\varphi|+|\nabla\varphi|^{p})|\varphi|\mathbbm{1}_{B }\,\dd x\\
 \geq & - c\int_\Omega\! (|\varphi|^{2}+|\nabla\varphi|^{2}+|\varphi|^{2p}+|\nabla\varphi|^{2p})\mathbbm{1}_{B } \,\dd x\\
  = & - c\int_\Omega\! \omega_L(|\varphi|+|\nabla\varphi|)(|\varphi|^2+|\nabla\varphi|^2+|\varphi|^{2p}+|\nabla\varphi|^{2p})\mathbbm{1}_{B }\,\dd x.
\end{align*}
For the last inequality before the equal sign we have used that, if $a,b>0$, then $ab\leq \frac{1}{2}(a^2+b^2)$ and, in addition, that if $1\leq q\leq s<\infty$, then for every $a>1$, it holds that $a^q\leq a^s$. On the other hand, the last identity follows from the fact that $\omega_L(t)=1 $ for $t\geq 1$. 

Compiling all the estimates above, and using that for $s\geq 1$ and $a,b>0$, $(a^s+b^s)\leq c_s(a+b)^s$, we can further obtain that, for a new constant $c>0$, 
\begin{align}\label{eq:PrefinalBMOlocmin}
& \int_\Omega \! \left( F(x,u+\varphi,\nabla u+\nabla\varphi)-F(x,u,\nabla u) \right) \,\dd x\notag \\
\geq &  c_0\int_\Omega\! |\nabla\varphi|^2\,\dd x - c\int_\Omega\! \omega_L(|\varphi|+|\nabla\varphi|)\left[( |\varphi|+|\nabla\varphi|)^2+( |\varphi|+|\nabla\varphi|)^{2p}\right]\,\dd x.
\end{align}
We now define the function $\Phi_s(t):=\omega_L(t)t^s$, with $s\in\{2,2p\}$. Note that, since $\Phi_s$ is increasing, then for every $a,b\geq 0$ we have that
\begin{equation}\label{eq:PhiAdit}
\Phi_s(a+b)\leq \Phi_s(2a)+\Phi_s(2b). 
\end{equation}
This implies that
\begin{equation}
\Phi_s(|\varphi|+|\nabla\varphi|)\leq \Phi_s(2|\varphi|)+\Phi_s(2|\nabla\varphi|). 
\end{equation}
On the other hand, by applying to the function $2\varphi$ Lemmata \ref{LemmaHLMaxVMO}, and \ref{lemmaFSMaxVMO}, Remark \ref{remark:PoincPhi} and then Lemma \ref{lemmaFSMaxVMO} one more time, and assuming that $\varphi$ takes the value of $0$ outside of $\Omega$, we whereby obtain that, for constants $c=c(n,N)>0$ and $\gamma=\gamma(n,p)>0$,  
\begin{equation}\label{eq:PoincApl}
\int_{\R^n} \! \Phi_s(2|\varphi|)\,\dd x\leq \gamma\int_{\R^n} \! \Phi_s(c(\nabla\varphi)^{\#})\,\dd x.
\end{equation}
 From \eqref{eq:PhiAdit}, \eqref{eq:PoincApl} and after applying Lemmata  \ref{LemmaHLMaxVMO} and \ref{lemmaFSMaxVMO} once again, but this time to the function $2\nabla\varphi$, we obtain that
 \begin{equation}\label{eq:BMOcontrolsp}
\int_{\R^n}\!\Phi_s( |\varphi| + |\nabla\varphi|  )\,\dd x \leq  \gamma \int_{\R^n} \! \Phi_s(c(\nabla\varphi)^{\#})\,\dd x.
\end{equation}
From inequalities \eqref{eq:PrefinalBMOlocmin} and \eqref{eq:BMOcontrolsp} we obtain that 
 \begin{align}
& \int_\Omega \! \left( F(x,u+\varphi,\nabla u+\nabla\varphi)-F(x,u,\nabla u) \right) \,\dd x\notag \\
 \geq &\,   c_0\int_\Omega\! |\nabla\varphi|^2\,\dd x - c\int_{\R^n}\! \omega_L(c(\nabla\varphi)^{\#})\left[( (\nabla\varphi)^{\#})^2+( (\nabla\varphi)^{\#})^{2p}\right]\,\dd x \label{eq:finalBMOlocmin}\\
 {\geq} & \, c_0\int_{\R^n}\! ((\nabla\varphi)^{\#})^2\,\dd x - c\int_{\R^n}\! \omega_L(c(\nabla\varphi)^{\#})\left[( (\nabla\varphi)^{\#})^2+( (\nabla\varphi)^{\#})^{2p}\right]\,\dd x. \notag 
\end{align}
For the second inequality above we have used \eqref{eq:f-f*} and \eqref{eqlemmaFSMax2VMO} once again. 

Note that, if $\delta_*\in (0,1)$ is such that $(\nabla\varphi)^{\#}<\delta_*$, the fact that $2p\geq 2$ will imply that that $((\nabla\varphi)^{\#})^{2p}<((\nabla\varphi)^{\#})^{2}$. Then, \eqref{eq:finalBMOlocmin} becomes
 \begin{align}
& \int_\Omega \! \left( F(x,u+\varphi,\nabla u+\nabla\varphi)-F(x,u,\nabla u) \right) \,\dd x\notag \\
 \geq & \, c_0\int_{\R^n}\! ((\nabla\varphi)^{\#})^2\,\dd x - c\int_{\R^n}\! \omega_L(c(\nabla\varphi)^{\#})( (\nabla\varphi)^{\#})^2\,\dd x. \notag 
\end{align}
Finally, since $\omega_L$ is continuous at $t=0$  and $\omega_L(0)=0$,  we can ensure that there exists $\delta_*\in (0,1)$ such that, if $0\leq t<\delta_*$, then
\[
c\omega_L(ct) < c_0.
\]
Whereby, for every $\varphi\in\WW^{1,\infty}_0(\Omega,\R^N)$ satisfying that $(\nabla\varphi)^{\#}<\delta_*$, we will have that, for a constant $c=c(n,N,p)>0$, 
\begin{equation*}
 \int_\Omega \! \left( F(x,u+\varphi,\nabla u+\nabla\varphi)-F(x,u,\nabla u) \right) \,\dd x \notag\\
\geq  \, c\int_{\R^n}\! ((\nabla\varphi)^{\#})^2\,\dd x .
\end{equation*}
Using again Lemmata \ref{LemmaHLMaxVMO} and \ref{lemmaFSMaxVMO}, we can further conclude that, for such $\varphi\in\WW^{1,\infty}_0(\Omega,\R^N)$, 
\begin{equation*}
 \int_\Omega \! \left( F(x,u+\varphi,\nabla u+\nabla\varphi)-F(x,u,\nabla u) \right) \,\dd x 
\geq  \, c\int_{\Omega}\! |\nabla\varphi|^2\,\dd x .
\end{equation*}
This concludes the proof of the theorem. 
\end{proof}

\section{The regularity result}\label{section:regularity}

The proof of Theorem \ref{theo:regularity} will follow in a classical way from the following decay estimate for the mean oscillations of the local minimizer $u$:
\begin{proposition}\label{Prop:FinalDecay}
Assume that $F$ and $u$ are as in Theorem \ref{theo:regularity}. For every $m>1$ there exist $\gamma\in(0,\beta)$, $\alpha\in(2\gamma,1)$, $R_\delta\in(0,1)$, $\tau_0\in (0,\frac{1}{4})$, and $\varepsilon_1\in (0,1)$, all depending on $n,N,L,F''$ and $m$, such that for every $\varepsilon\in(0,\varepsilon_1)$, if $\varrho\in (0,R_\delta)$, $|(u)_{x_0,\varrho}|+|(\nabla u)_{x_0,\varrho}|<m$ and $E(x_0,\varrho)<\varepsilon$, then
\begin{equation}
E(x_0,\tau_0\varrho) \leq c_0 R^{2\gamma}+\tau_0^\alpha E(x_0,\varrho). 
\end{equation}
\end{proposition}

There are several steps to follow in order to achieve this decay estimate. Furthermore, we recall that the minimality assumption is essential  for a regularity result of this nature, since maps that are merely solutions to the weak Euler-Lagrange equations associated to the problem, need not satisfy the stated partial regularity property \cite{KT,MS03}. 

In order to make use of the local minimality condition, we will need to build suitable test functions that are sufficiently small in $\WW^{1,q}(\Omega,\R^N)$. With this in mind, we make the following observation.

\begin{lemma}\label{remark:varphi}
Take $q\in[1,\infty]$ and assume that $u\in\WW^{1,q}_{\mathrm{loc}}(\Omega,\R^N)$. Let $B_R= B(x_0,R)$, with $R\leq 1$, and suppose that $B_r=B(y_0,r)\subseteq B_R$. Furthermore, let $\zeta\in (0,1)$ and assume that $\rho$ is a cut-off function such that 
\[
\mathbbm{1}_{B_{\zeta r}}\leq \rho\leq \mathbbm{1}_{B_r} \,\,\mbox{ and }\,\,|\nabla\rho|\leq \frac{1}{r(1-\zeta)}. 
\]
 In addition, for a given $m>0$, let $z_0:=(\nabla u)_{B_R}\in\R^{N\times n}$ and assume that $|z_0|\leq m$. Finally,  let $a\colon\R^n\to\R^N$ be an affine function of the form $a(x):=z_0\cdot(x-y_0)+(u)_{y_0,r}$ and define
\[
\varphi:=\rho(u-a).
\]
If $q\in[1,\infty)$, then for any given $\delta>0$ there exist a constant $c>0$, and a radius $ R_\delta=R_\delta(c,\zeta,m,q)\in(0,1)$ such that, if $0<R<R_\delta$ and $B_r\subseteq B_R$, then
\begin{equation}\label{eq:philessdelta}
\|\nabla\varphi  \|_{\LL^q(\Omega,\R^{N\times n})}<\delta. 
\end{equation} 
On the other hand, if $q=\infty$ and $\delta>0$ is given, there exists $\delta_0=\delta_0(\delta,c,\zeta,m,q)>0$ such that, if  $\delta_1\in(0,\delta_0)$ and \eqref{eq:q=infty} holds, then we can find a radius $ R_\delta=R_\delta(c,\zeta,m,q)\in(0,1)$ with the property that, if $0<R<R_\delta$ and $B_r\subseteq B_R$, then
\begin{equation}\label{eq:philessdeltainfty}
\|\nabla\varphi  \|_{\LL^\infty(\Omega,\R^{N\times n})}<\delta. 
\end{equation} 
\end{lemma}
\begin{proof}
If $1\leq q<\infty$, then Poincar\'e inequality implies that there is a constant $c=c(n,N,q,m)$ such that
\begin{equation}\label{eq:Testq<infty2}
\|\nabla\varphi  \|_{\LL^q(B_r,\R^{N\times n})}\leq c\left( 1+\frac{1}{1-\zeta} \right) \|\nabla u-z_0\|_{\LL^q(B_r,\R^{N\times n})}.
 \end{equation}
If $B_r\subseteq B_R$, it clearly follows that 
 \begin{equation}\label{eq:Testq<infty}
\|\nabla\varphi  \|_{\LL^q(B_r,\R^{N\times n})}\leq c\left( 1+\frac{1}{1-\zeta} \right) \|\nabla u-z_0\|_{\LL^q(B_R,\R^{N\times n})}.
 \end{equation}
Note that, under the assumptions that  $u\in\WW^{1,q}_{\mathrm{loc}}(\Omega,\R^N)$ and $|z_0|\leq m$, for a given $\delta>0$ we will be able to make $\|\nabla\varphi  \|_{\LL^q(\Omega,\R^{N\times n})}<\delta$ by taking $B(y_0,r)\subseteq B(x_0,R)$, with $R>0$ sufficiently small. This concludes the proof  of \eqref{eq:philessdelta} for the case $q\in[1,\infty)$. 
 
If $q=\infty$, our aim is to prove that, if $B_r\subseteq B_R$, then for a constant $c=c(N)>0$, 
 \begin{equation}\label{eq:TestInfty}
\|\nabla\varphi  \|_{\LL^\infty(B_r,\R^{N\times n})}\leq c\left( 1+\frac{1}{1-\zeta} \right) \|\nabla u-z_0\|_{\LL^\infty(B_R,\R^{N\times n})}.
 \end{equation}
 In order to do this, we will first show that 
  \begin{equation}\label{eq:TestInfty2}
\|\nabla\varphi  \|_{\LL^\infty(B_r,\R^{N\times n})}\leq c\left( 1+\frac{1}{1-\zeta} \right) \|\nabla u-z_0\|_{\LL^\infty(B_r,\R^{N\times n})}.
 \end{equation}
 Inequality \eqref{eq:TestInfty} will then follow after taking now the $\LL^\infty$  norm from the right hand side on the larger ball $B_R$. 
 
In order to show \eqref{eq:TestInfty2}, we follow the ideas for the calculations made in \cite{KT}. We deal with each coordinate function of $u-a$ and, for a fixed $1\leq k\leq N$, we take $x_k\in B(y_0,r)$ and $y_k\in \overline{B(y_0,r)}$ such that:
\begin{itemize}
\item $(u-a)^{(k)}(x_k)=(u-a)^{(k)}_{y_0,r}=0$;
\item $|(u-a)^{(k)}(y_k)|=\sup_{B(y_0,r)} |(u-a)^{(k)}|$.
\end{itemize}
Then, by the Fundamental Theorem of Calculus applied over the segment $[x_k,y_k]$ to each coordinate function, we obtain that
\begin{align*}
\sup_{B(y_0,r)} |u-a|\leq &\, N\max_{1\leq k\leq N}\left|  \int_{x_k}^{y_k} \nabla (u-a)^{(k)}  \right|\\
\leq & \,2Nr\,\|\nabla (u-a)\|_{\LL^{\infty}(B_r,\R^{N\times n})}\\
= & \, 2Nr \,\|\nabla u-z_0\|_{\LL^{\infty}(B_r,\R^{N\times n})}.
\end{align*}
This readily implies \eqref{eq:TestInfty2}.

On the other hand,  if $\delta>0$ is given and we take $B(y_0,r)\subseteq B(x_0,R)$ and $z_0:=(\nabla u)_{B_R}$,  inequality \eqref{eq:TestInfty} implies  that, if $0<\delta_1< \tfrac{\delta}{c(1+1/(1-\zeta))}$ and \eqref{eq:q=infty} holds,  there exists  $R_\delta=R_\delta(c,\zeta,m,q)>0$ such that, if $0<R<R_\delta$, then
\begin{equation*}
\|\nabla\varphi  \|_{\LL^\infty(\Omega,\R^{N\times n})}=\|\nabla\varphi  \|_{\LL^\infty(B_r,\R^{N\times n})}\leq  c\left( 1+\frac{1}{1-\zeta} \right) \|\nabla u-z_0\|_{\LL^\infty(B_R,\R^{N\times n})}<\delta, 
\end{equation*} 
and  \eqref{eq:philessdeltainfty} holds.
\end{proof}
We remark that, in what follows, $\zeta\in(0,1)$ will be fixed to be a suitable constant depending on other fixed parameters, none of which will depend on $\delta$.

\subsection{Higher integrability}\label{Sec:HigherInt}
The next step is to establish a preliminary higher integrability result. For the rest of this section we assume that $F$ satisfies conditions \ref{H0}-\ref{H4:HoldCont} for some $p\geq 2$.

\begin{lemma}[Preliminary higher integrability]\label{PrelHInt}
Let $q\in[1,\infty]$ and let $u\in\WW^{1,p}(\Omega,\R^N)\cap \WW^{1,q}_{\mathrm{loc}}(\Omega,\R^N)$ be a $\WW^{1,q}$-local minimizer of $\mathcal{F}$. Furthermore, if $q=\infty$, assume that \eqref{eq:q=infty} holds for $\delta_1\in (0,\delta_0)$, with $\delta_0$ as the one given by Lemma \ref{remark:varphi}. Then, for every $m>0$ there exists some $R_\delta=R_\delta(m,n,N,q,L)>0$ such that, if $2R_0\in(0,R_\delta)$, $B_{2R_0}=B(x_0,2R_0)\subseteq \Omega$ is a ball, and $|(\nabla u)_{x_0,2R_0}|\leq m$, then there exist $\varepsilon_0>0$ and $c=c(m,n,N,L)>0$ such that, for every $q_0\in [p,p+\varepsilon_0)$, 
\begin{align*}
\left( \dashint_{B_{\frac{1}{2} R_0}} |\nabla u|^{q_0}\, \dd x \right)^{\frac{1}{q_0}}\leq  c \,\left( \dashint_{B_{R_0} }|\nabla u|^{p } \,\dd x \right)^{\frac{1}{p}}+ c.
\end{align*}
\end{lemma}
Before proceeding to proving the lemma, we remark that the need to make the assumption that it is for the radius $2R_0$ that it holds $|(\nabla u)_{x_0,2R_0}|\leq m$ will become clear while establishing \eqref{III} in the proof of Theorem \ref{theo:linearization}. Furthermore, fixing this constant vector at this stage will be necessary in order to use the local minimality condition for the case $q=\infty$. We shall elaborate further on this at the end of the following proof. 
\begin{proof}
Let $B_{R_0}=B(x_0,R_0)\subseteq \Omega$ be a ball of radius $R_0$ and let $B_r=B(y,r)\subseteq B_{R_0}$. 

We define  $z_0:=(\nabla u)_{x_0,2R_0}$ and assume that it satisfies $|z_0|\leq m$. Furthermore, we let 
\[
a(x):=z_0\cdot(x-y)+(u)_{y,r}.
\] 
On the other hand, for a $\zeta\in (0,1)$ yet to be determined, and that will depend exclusively on a constant $\tilde{\theta}=\tilde{\theta}(n,L)>0$, we let $\rho$ be a cut-off function such that 
\[
\mathbbm{1}_{B_{\zeta r}}\leq \rho\leq \mathbbm{1}_{B_r} \,\,\mbox{ and }\,\,|\nabla\rho|\leq \frac{1}{r(1-\zeta)}. 
\]
Define
\[
\varphi:=\rho(u-a), \hspace{1cm} \psi:=(1-\rho)(u-a).
\]

Then, by the strong quasiconvexity of $G$ and assumption \ref{H3:GleqF}, which at this step turns out to be crucial, we obtain that
\begin{align}
 \int_{B_r} \left( |\nabla\varphi|^p + {G(z_0)} \right) \,\dd x \leq & \int_{B_r} G(z_0+\nabla\varphi)\,\dd x\notag \\
\leq  & \int_{B_r} {F}(x,u,z_0+\nabla \varphi)\dd x  \label{HI1} \\
= & \int_{B_r} {F}(x,u,\nabla u -\nabla\psi )\,\dd x  \notag \\
=  & \int_{B_r} {F}(x,u,\nabla u)\,\dd x \notag\\
& + \int_{B_r} \left( {F}(x,u,\nabla u -\nabla\psi )-  {F}(x,u,\nabla u) \right) \,\dd x . \notag 
\end{align}
To estimate the first term on the right hand side we use the local minimality of $u$. Indeed, if $\|\nabla\varphi\|_{\LL^q(\Omega,\R^{N\times n})}<\delta$, then 
\begin{align}
& \int_{B_r} F(x,u,\nabla u) \,\dd x  \notag \\
\leq & \int_{B_r} F(x,u-\varphi, \nabla u -\nabla\varphi)\,\dd x \notag \\
= & \int_{B_r} F(x,\psi + a,  \nabla \psi + \nabla a)\, \dd x \notag \\
= & \int_{B_r\backslash B_{\zeta r}} F(x,\psi + a,\nabla \psi + z_0)\,\dd x + \int_{B_{\zeta r}}F(x,a,z_0)\,\dd x \label{HI2} \\
\stackrel{\mathrm{\ref{H1:pg}}}\leq & 
L \int_{B_r\backslash B_{\zeta r} }|\nabla\psi+z_0|^p \,\dd x + cr^n \notag\\
\leq & \, c \int_{B_r\backslash B_{\zeta r}} \left( |\nabla u- z_0|^p + \frac{|u-a|^p}{r^p(1-\zeta)^p} \right) \,\dd x  + cr^n. \notag 
\end{align}
Here, $c=c(L,p,m)$. On the other hand, by \eqref{LipFz}, and taking into account that $\psi=0$ in $B_{\zeta r}$ and that $|z_0|\leq m$, we obtain that 
\begin{align}
&\int_{B_r} \left( {F}(x,u,\nabla u -\nabla\psi)-  {F}(x,u,\nabla u) \right) \,\dd x \notag \\ \leq & \, c \int_{B_r\backslash B_{\zeta r}} \left( 1+|\nabla u |^{p-1} + |\nabla \psi|^{p-1} \right) |\nabla\psi|\,\dd x \label{HI3} \\
 \leq &\, c\int_{B_r\backslash B_{\zeta r}} \left( |\nabla u|^p + \frac{|u-a |^p}{r^p(1-\zeta)^p} \right) \,\dd x + cr^n. \notag 
\end{align}
From \eqref{HI1}-\eqref{HI3} and using that $G$ is locally bounded, we deduce that
\begin{align*}
\int_{B_{\zeta r}} |\nabla u|^p\, \dd x\leq c\int_{B_r\backslash B_{\zeta r}} \left( |\nabla u|^p + \frac{|u-a|^p}{r^p(1-\zeta)^p} \right) \,\dd x + cr^n.
\end{align*}
We now use Wydman's hole-filling strategy to obtain that, for some $\tilde{\theta}\in (0,1)$, with $\tilde{\theta}$ depending exclusively on $m, L$ and $p$, the following holds:
\begin{align*}
\int_{B_{\zeta r}} |\nabla u|^p\, \dd x\leq \tilde{\theta} \int_{B_r } \left( |\nabla u|^p + \frac{|u-a|^p}{r^p(1-\zeta)^p} \right) \,\dd x + \tilde{\theta}r^n.
\end{align*}
Therefore, if we fix $\tilde{\theta}<\zeta^n$, for $\theta:=\frac{\tilde{\theta}}{\zeta^n}\in(0,1)$ we have that

\begin{align*}
\dashint_{B_{\zeta r}} |\nabla u|^p\, \dd x\leq {\theta} \,\dashint_{B_r } \left( |\nabla u|^p + \frac{|u-a|^p}{r^p(1-\zeta)^p} \right) \,\dd x + {\theta}.
\end{align*}
We now apply Poincar{\'e}-Sobolev inequality to the second term on the right hand side  and  use that $|z_0|\leq m$ to obtain that 
\begin{align}\label{eq:preGehringPrelHighInt}
\dashint_{B_{\zeta r}} |\nabla u|^p\, \dd x\leq {\theta} \,\dashint_{B_r } |\nabla u|^p\,\dd x  + \frac{c({\theta},n)}{(1-\zeta)^p} \,\left( \dashint_{B_r }|\nabla u|^{\frac{pn}{n+p}} \,\dd x \right)^{\frac{n+p}{n}}+ c(m,\theta).
\end{align}
We note that this will all hold as long as $\|\nabla\varphi\|_{\LL^q(\Omega,\R^{N\times n})}<\delta$. In order to achieve this, we use Lemma \ref{remark:varphi} (with $R=2R_0$) and the fact that $u\in\WW^{1,q}_{\mathrm{loc}}(\Omega,\R^N)$ to conclude that, for our fixed choice of $\zeta=\zeta(m,n,L)$, we can find $R_\delta=R_\delta(m,n,N,q,L)\in(0,1)$ such that,  if $q\in [1,\infty]$, then for all $0<R_0<R_\delta$, if $B_r\subseteq B(x_0,R_0)$, then $\|\nabla\varphi\|_{\LL^q(\Omega,\R^{N\times n})}<\delta$.

We remark that, for the case $q=\infty$, we are using here that $z_0=(\nabla u)_{x_0,2R_0}$ and that \eqref{eq:q=infty} holds in order to use the estimate given by  Lemma \ref{remark:varphi}.

Whereby, considering that inequality \eqref{eq:preGehringPrelHighInt} holds for every $B_r\subseteq B(x_0,R_0)$, we can finally apply Gehring's Lemma in the version of Theorem \ref{theo:Gehring} and obtain that there exist $\varepsilon_0=\varepsilon_0(m,n,\theta,p)>0$ and $c=c(m,n,\theta,p)>0$ such that, for $q_0\in [p,p+\varepsilon_0)$, 
\begin{align*}
\left( \dashint_{B_{\frac{R_0}{2}}} |\nabla u|^{q_0}\, \dd x \right)^{\frac{1}{q_0}}\leq  c \,\left( \dashint_{B_{R_0} }|\nabla u|^{p } \,\dd x \right)^{\frac{1}{p}}+ c.
\end{align*}
This concludes the proof of the Lemma. 
\end{proof}
\subsection{A Caccioppoli inequality of the first kind}\label{Sec:Caccioppoli}
We are ready to establish a Caccioppoli inequality that, combined with the higher integrability from Lemma \ref{PrelHInt}, will pave the way to finally obtain the decay in the mean oscillations that we pursue. For the results in this section we also assume that $F$ satisfies the conditions \ref{H0}-\ref{H4:HoldCont} for some $p\geq 2$.

The Caccioppoli inequality will fundamentally rely on the local minimality property satisfied by $u$. For this reason, and considering the comments made in Lemma \ref{remark:varphi}, inequality \eqref{Caccioppoli_Statement} will be valid for a fixed $\zeta\in(0,1)$, that will depend on other fixed parameters. This implies that, just as in \cite{KT}, and opposite to what happens in the case of global minimizers, as in \cite{Eva86,GiaqGius}, we cannot iterate \eqref{Caccioppoli_Statement} in order to obtain a Caccioppoli inequality of the first kind. 

As discussed in the introduction, Kristensen and Taheri overcame this problem in the case of homogeneous integrands by means of a measure theoretical compactness argument for the blown-up sequence. However, given the need of a direct argument that is brought by the dependence on $u$ of $F$, we will instead make use of Theorem \ref{theo:Gehring} to derive from \eqref{Caccioppoli_Statement} a reverse H\"older inequality in Theorem \ref{theo:ReverseHolder}.
\begin{theorem}\label{theo:Caccioppoli}
Let $q\in[1,\infty]$ and let $u\in\WW^{1,p}(\Omega,\R^N)\cap \WW^{1,q}_{\mathrm{loc}}(\Omega,\R^N)$ be a $\WW^{1,q}$-local minimizer of $\mathcal{F}$. If $q=\infty$, assume further that \eqref{eq:q=infty} holds for some $\delta_1\in(0,\delta_0)$, with $\delta_0>0$ given by Lemma \ref{remark:varphi}. 

Then, for every $m>0$ there exist $R_\delta=R_\delta(n,N,q,L,m)\in(0,1)$, as well as $\zeta,\theta\in(0,1)$, with $\zeta=\zeta(m,n,L,F'')$ and $\theta=\theta(m,n,L,F'')$, such that, if $R_0\in(0,R_\delta)$, $B(x_0,2R_0)\subseteq\Omega$,  $z_0:=(\nabla u)_{x_0,2R_0}$  and $u_0\in\R^{N}$ satisfy $|u_0|+|z_0|\leq m$, then  for every  $B(y_0,r)\subseteq B(x_0,{R_0})$ and for $a\colon\R^n\to\R^N$ the affine map given by $a(x):=z_0\cdot(x-y_0)+(u)_{y_0,r}$,  we have that
\begin{align}
&  \dashint_{B_{\zeta r}} |V(\nabla u -z_0)|^2 \,\dd x \notag \\
\leq &\, {\theta}\, \dashint_{B_r }  |V( \nabla u - z_0) |^2  \, \dd  x  + {\theta}\, \dashint_{B_r } \left | V\left(\frac{u-a }{r(1-\zeta)} \right) \right|^2 \, \dd x \label{Caccioppoli_Statement} \\
& + {\theta}\, \dashint_{B_r}  \vartheta( |u_0|, 2 |x-y_0|^2 + 2 |u-u_0|^2  ) \left(1 + |\nabla u|^p + |z_0|^p\right)\, \dd x \notag \\
& + {\theta}\, \dashint_{B_r }  \vartheta( |u_0|, 2|u-a|^2  ) \left(1 + |z_0|^p \right)  \, \dd x.\notag 
\end{align}
\end{theorem}

\begin{proof}[Proof of Theorem \ref{theo:Caccioppoli}]
Let $m>0$ be arbitrary and fix $(x_0,u_0,z_0)\in\Omega\times\R^N\times\R^{N\times n}$, with $|u_0|+|z_0|\leq m$. Take $B_r:=B(y_0,r)\subseteq B(x_0,{R_0})=B_{R_0}$, with $R_0$ to be determined. 

Furthermore, let $\zeta\in (0,1)$ and assume that $\rho$ is a cut-off function satisfying that
\[
\mathbbm{1}_{B_{\zeta r}}\leq \rho\leq \mathbbm{1}_{B_r} \,\,\,\mbox{ and }\,\,\,|\nabla\rho|\leq \frac{1}{r(1-\zeta)}. 
\]
Define $a(x):=z_0\cdot(x-y_0)+(u)_{y_0,r}$, 
\[
\varphi:=\rho(u-a), \hspace{0.5cm}\mbox{ and }\hspace{0.5cm} \psi:=(1-\rho)(u-a).
\]
In addition, let $\bar{F}(z):=F(y_0,u_0,z_0+z)-F(y_0,u_0,z_0)-F_z(y_0,u_0,z_0)[z]$. 
\\
Then, by first using \ref{H2:QC} and then Lemma \ref{lemma:FrozenShift}, we obtain the following:
\begin{align*}
  \int_{B_r} |V(\nabla\varphi)|^2 \,\dd x 
\leq & \int_{B_r} \bar{F}(\nabla\varphi)\,\dd x 
=  \int_{B_r} \bar{F}(\nabla u - z_0 - \nabla\psi)\,\dd x\\
= & \int_{B_r} \bar{F}(\nabla u-z_0)\,\dd x + \int_{B_r} \left( \bar{F}(\nabla u-z_0-\nabla\psi) - \bar{F}(\nabla u-z_0) \right)\,\dd x\\
\leq & \int_{B_r} \bar{F}(\nabla u-z_0)\,\dd x  \\
& +c\int_{B_r}\left( |\nabla\psi|+|\nabla u - z_0|+|\nabla \psi|^{p-1}+|\nabla u - z_0|^{p-1} \right)|\nabla\psi|\,\dd x.  
\end{align*}
For the first term on the right hand side above, we note that
\begin{align*}
& \int_{B_r}\bar{F}(\nabla u - z_0)\,\dd x \\
= & \int_{B_r}\big( F(y_0,u_0,\nabla u)-F(y_0,u_0,z_0)-F_z(y_0,u_0,z_0)[\nabla u - z_0] \big)\,\dd x\\
= & \int_{B_r}F(x,u,\nabla u)\,\dd x +\int_{B_r}  \big( F(y_0,u_0,\nabla u)- F(x,u,\nabla u) \big) \,\dd x\\
& -\int_{B_r} \big( F(y_0,u_0,z_0)+F_z(y_0,u_0,z_0)[\nabla u - z_0] \big) \,\dd  x.
\end{align*}
On the other hand, if $\|\nabla\varphi\|_{\LL^q}<\delta$, then the local minimality property  of $u$ implies that
\begin{align*}
 \int_{B_r} F(x,u,\nabla u) \,\dd x 
\leq & \int_{B_r} F(x,u-\varphi, \nabla u -\nabla\varphi)\,\dd x\\
 = & \int_{B_r}F(y_0,u_0,\nabla\psi+z_0)\,\dd x\\
& + \int_{B_r} \left(F(x,\psi+ a ,\nabla\psi+z_0) - F(y_0,u_0,\nabla\psi + z_0) \right)\,\dd x\\
= & \int_{B_r}\bar{F}(\nabla\psi)\,\dd x + \int_{B_r}F(y_0,u_0,z_0)\,\dd x + \int_{B_r}F_z(y_0,u_0,z_0)[\nabla\psi]\,\dd x\\
& + \int_{B_r}\left( F(x,\psi+ a ,\nabla\psi+z_0) - F(y_0,u_0,\nabla\psi+z_0)   \right) \,\dd x.
\end{align*}
Compiling all the estimates above we obtain that
\begin{align}
&  \int_{B_r} |V(\nabla\varphi)|^2\,\dd x \notag \\
\leq & \int_{B_r} \bar{F}(\nabla \psi )\,\dd x + c\int_{B_r}\left(|\nabla\psi|+|\nabla u - z_0|+|\nabla \psi|^{p-1}+|\nabla u - z_0|^{p-1} \right)|\nabla\psi|\,\dd x \notag \\
& + \int_{B_r} \left(F(y_0,u_0,\nabla u)-F(x,u,\nabla u)     \right) \,\dd x \notag  \\
& + \int_{B_r}F_z(y_0,u_0,z_0)[\nabla\psi]\,\dd x - \int_{B_r}F_z(y_0,u_0,z_0)[\nabla\varphi+\nabla\psi]\,\dd x \notag \\ 
& + \int_{B_r}\big( F(x,\psi+ a ,\nabla\psi+z_0)-F(y_0,u_0,\nabla\psi + z_0)  \big) \,\dd x.\notag
\end{align}
Using Lemma \ref{lemma:FrozenShift} for the first term and Young's inequality for the second term on the right hand side, together with $\varphi\in\WW^{1,p}_0(B_r,\R^N)$ and the definition of $V$, this implies that, for $\zeta\in(0,1)$, 
\begin{align}
 \int_{B_{\zeta r}} |V(\nabla u -z_0)|^2 \,\dd x 
\leq &\, c\int_{B_r\backslash B_{\zeta r}} |V(\nabla\psi)|^2  \, \dd x \notag \\
& + c\int_{B_r\backslash B_{\zeta r}}  |V(\nabla u - z_0)|^2   \, \dd x\notag \\
& + \int_{B_r} \left(F(y_0,u_0,\nabla u)-F(x,u,\nabla u)     \right) \,\dd x \label{MinForCacc}  \\
& + \int_{B_r}\big( F(x,\psi+ a ,\nabla\psi+z_0)-F(y_0,u_0,\nabla\psi + z_0)  \big) \,\dd x.\notag
\end{align}
Now we use assumption \ref{H4:HoldCont} and the fact that $\varphi\in\WW^{1,p}_0(B_r,\R^N)$ to obtain from above that
\begin{align*}
&  \int_{B_{\zeta r}} |V(\nabla u -z_0)|^2 \,\dd x \notag\\
\leq &\, c\int_{B_r\backslash B_{\zeta r}} |V(\nabla\psi)|^2  \, \dd x + c\int_{B_r\backslash B_{\zeta r}}  |V(\nabla u - z_0)|^2   \, \dd x\\
& + c\int_{B_r}  \vartheta( |u_0|, |x-y_0|^2 + |u-u_0|^2  ) \left(1 + |\nabla u|^p \right)\, \dd x\\
& + c\int_{B_r}  \vartheta( |u_0|, |x-y_0|^2 + |u-u_0|^2 + |u- a |^2 ) \left(1 + |\nabla \psi|^p + |z_0|^p \right)\, \dd x.
\end{align*}
Since $\vartheta\leq 2 L$, $\psi=0$ in $B_{\zeta r}$, and $|\cdot|^p\leq |V|^2$,  this implies that
\begin{align*}
&  \int_{B_{\zeta r}} |V(\nabla u -z_0)|^2 \,\dd x \\
\leq &\, c\int_{B_r\backslash B_{\zeta r}} |V(\nabla\psi)|^2  \, \dd x
 + c\int_{B_r\backslash B_{\zeta r}} |V( \nabla u - z_0) |^2 \, \dd x\\
& + c\int_{B_r}  \vartheta( |u_0|, |x-y_0|^2 + |u-u_0|^2  ) \left(1 + |\nabla u|^p  \right)\, \dd x\\
& + c\int_{B_r}  \vartheta( |u_0|, |x-y_0|^2 + |u-u_0|^2 + |u-a|^2 ) \left(1 + |z_0|^p \right)\, \dd x.
\end{align*}
Note that we have made a hole for all the terms where $|\nabla\psi|$ appears as a factor. We emphasize that performing this step is crucial to ensure that we will later be able to suitably apply the higher integrability from Lemma \ref{PrelHInt}.
Observe here that, since $\vartheta$ is increasing and non-negative, then for all $s>0$ and for all $a,b>0$, 
\begin{equation}\label{eq:vartheta-subadit}
\vartheta(s,a+b)\leq \vartheta(s,2a)+\vartheta(s,2b).
\end{equation}
We now use the definition of $\psi$ in the inequality above, \eqref{eq:vartheta-subadit} and that $\vartheta$ is increasing, to obtain that
\begin{align*}
&  \int_{B_{\zeta r}} |V(\nabla u -z_0)|^2  \,\dd x \\
\leq &\, c\int_{B_r\backslash B_{\zeta r}}  |V(\nabla u - z_0)|^2 \, \dd x 
 + c\int_{B_r} \left|V\left(\frac{u-a}{r(1-\zeta)}\right)\right|^2  \, \dd x\\
& + c\int_{B_r}  \vartheta( |u_0|, 2 |x-y_0|^2 + 2 |u-u_0|^2  ) \left(1 + |\nabla u|^p + |z_0|^p\right)\, \dd x\\
& + c\int_{B_r }  \vartheta( |u_0|, 2|u-a|^2  ) \left(1 + |z_0|^p \right)  \, \dd x.
\end{align*}
We fill in the hole and conclude that, for some $\tilde{\theta}\in(0,1)$, 
\begin{align*}
& \int_{B_{\zeta r}} |V(\nabla u -z_0)|^2  \,\dd x \\
\leq &\,  \tilde{\theta}\int_{B_r }  |V(\nabla u - z_0)|^2  \, \dd x 
  +\tilde{\theta} \int_{B_r } \left| V\left(\frac{u-a}{r(1-\zeta)} \right) \right|^2 \, \dd x\\
& + \tilde{\theta}\int_{B_r}  \vartheta( |u_0|, 2 |x-y_0|^2 + 2 |u-u_0|^2  ) \left(1 + |\nabla u|^p + |z_0|^p\right)\, \dd x\\
& + \tilde{\theta}\int_{B_r }  \vartheta( |u_0|, 2|u-a|^2  ) \left(1 + |z_0|^p \right)  \, \dd x.
\end{align*}
Henceforth, if we assume that $\tilde{\theta}<\zeta^n$, this implies that, for $\theta =\tilde{\theta}/\zeta^n\in (0,1)$, 
\begin{align*}
 & \dashint_{B_{\zeta r}} |V(\nabla u -z_0)|^2  \,\dd x \\
\leq &\,  {\theta}\,\dashint_{B_r }  |V(\nabla u - z_0)|^2  \, \dd x\\
& +{\theta} \,\dashint_{B_r } \left| V\left(\frac{u-a}{r(1-\zeta)} \right) \right|^2 \, \dd x\\
& + {\theta}\,\dashint_{B_r}  \vartheta( |u_0|, 2 |x-y_0|^2 + 2 |u-u_0|^2  ) \left(1 + |\nabla u|^p + |z_0|^p\right)\, \dd x\\
& + {\theta}\,\dashint_{B_r }  \vartheta( |u_0|, 2|u-a|^2  ) \left(1 + |z_0|^p \right)  \, \dd x.
\end{align*}
We recall that this will hold provided we can use the $\WW^{1,q}$-local minimality of $u$ to obtain \eqref{MinForCacc}, meaning that we need $\|\nabla\varphi  \|_{\LL^q(\Omega,\R^{N\times n})}<\delta$. By Lemma \ref{remark:varphi} and the fact that $u\in\WW^{1,q}_{\mathrm{loc}}(\Omega,\R^N)$, as well as the assumptions that $|z_0|\leq m$ and \eqref{eq:q=infty}, we can argue as at the end of Lemma \ref{PrelHInt} and, whereby, find $R_\delta=R_\delta(n,N,q,L,m)\in(0,1)$ such that, for every $r\in (0,R_\delta)$, $\|\nabla\varphi\|_{\LL^q(\Omega,\R^{N\times n})}<\delta$ and, hence, \eqref{Caccioppoli_Statement} holds. 
\end{proof}
\subsection{A reverse H\"older inequality}\label{Sec:RevHolder}
We will now use the previous Caccioppoli inequality to obtain a  reverse H\"{o}lder inequality.  It is important to note at this stage that, in order to do so, we need to obtain an expression where the term $$\theta\dashint_{B_r}\! |V(\nabla u-z_0)|^2\,\dd x$$ is improved into an estimate with a higher exponent. It is for this purpose that we wish to apply Theorem \ref{theo:Gehring}. Once again, for this section we assume that $F$ satisfies the conditions \ref{H0}-\ref{H4:HoldCont} for some $p\geq 2$.

On the other hand, we note that in \eqref{Caccioppoli_Statement}, we also have the lower order term $u-a$ on the right hand side. This is already a higher integrable term, by Poincar\'e-Sobolev compactness theorem. However, in order to take advantage of this, we have made $a$ dependent on $r>0$ in the statement of Theorem \ref{theo:Caccioppoli}. It is therefore crucial that both appearances of $u-a$ have been made so that they are only being multiplied by constants and, whereby, we can readily apply Poincar\'e-Sobolev inequality. This allows us to obtain an estimate where the integrands do not depend anymore on $r>0$ and, finally, apply Gehring's Lemma in its generalized version of Theorem \ref{theo:Gehring}. 
We proceed to implement these ideas in order to obtain the following result.

\begin{theorem}[A reverse H\"older inequality]\label{theo:ReverseHolder}
Let $q\in[1,\infty]$ and let  $u\in\WW^{1,p}(\Omega,\R^N)\cap \WW^{1,q}_{\mathrm{loc}}(\Omega,\R^N)$ be a $\WW^{1,q}$-local minimizer of $\mathcal{F}$. Furthermore, if $q=\infty$, assume that \eqref{eq:q=infty} holds for some $\delta_1>0$ taken as in Lemma \ref{remark:varphi}. Then,  for every $m>0$ there exist $R_\delta=R_\delta(n,N,q,L,m)\in (0,1)$, a constant  $c=c(m,n,L,F'')>0$ and  $q_1>2$,  such that, if $2 R_0<R_\delta$, $B_{2 R_0}=B(x_0,2R_0)\subseteq \Omega$ is a ball,  and  $|(u)_{2R_0}|+|(\nabla u)_{2R_0}|\leq m$, then we have that, for $z_0=(\nabla u)_{2R_0}$, 
\begin{align}
& \left( \dashint_{B_{\frac{R_0}{2}}} |V(\nabla u -z_0)|^{{q_1}} \,\dd x \right)^{\frac{2}{q_1}} \notag \\
\leq & \, c\,\dashint_{B_{2 R_0} }  |V( \nabla u - z_0) |^2  \, \dd x \label{eq:RevHold} \\
& + c\, \vartheta_0 \left( |(u)_{x_0,2R_0}|,  R_0^2 + R_0^2\,\dashint_{B_{2 R_0}}  |\nabla u - z_0|^2  \,\dd x \right)
  + c  \,R_0^{2\beta}. \notag
\end{align}
Here, $\vartheta_0(u,t):=\vartheta^{\kappa}(u,ct)$ for a constant $c=c(m,n)>0$ and an exponent $\kappa=\kappa(n,p,\theta,m)\in(0,1)$, with $\theta$ as in Theorem \ref{theo:Caccioppoli}.  
\end{theorem}

\begin{proof}
Let $m>0$, and let $R_\delta,\zeta,\theta\in (0,1)$ be as in Theorem \ref{theo:Caccioppoli}. Take $R_0\in (0,\frac{1}{2}R_\delta)$ and a ball $B_{2 R_0}=B(x_0,2R_0)\subseteq \Omega$. 

Define $u_0:=(u)_{2R_0}$, $z_0:=(\nabla u)_{2R_0}$, and assume that $(u_0,z_0)\in\R^N\times \R^{N\times n}$ are such that 
\[
|u_0|+|z_0|\leq m.
\]
Furthermore, let $B_r=B(y_0,r)\subseteq B_{R_0}$. 
\\
Now, use that $\vartheta$ is concave for the last term on the right hand side of \eqref{Caccioppoli_Statement}, and then Poincar{\'e}-Sobolev inequality for all the appearances of $u-a$, as well as the increasing nature of $\vartheta$. We henceforth obtain that 
\begin{align}
\dashint_{B_{\zeta r}} |V(\nabla u -z_0)|^2 \,\dd x \notag 
\leq &\, {\theta}\, \dashint_{B_r }  |V( \nabla u - z_0) |^2  \, \dd x \notag \\
& + c \, \left( \dashint_{B_r } \left | V\left(\nabla u - z_0 \right) \right|^\frac{2n}{n+2} \, \dd x\right)^{\frac{n+2}{n}} \label{Cacc1} \\
& + {\theta}\, \dashint_{B_r}  \vartheta( |u_0|, 2 |x-y_0|^2 + 2 |u-u_0|^2  ) \left(1 + |\nabla u|^p + |z_0|^p\right)\, \dd x\notag \\
& + {\theta} \left(1 + |z_0|^p \right)  \,\vartheta\left(|u_0|, 2cr^2 \left(\dashint_{B_r }   |\nabla u - z_0|^{\frac{2n}{n+2}}    \, \dd x \right)^{\frac{n+2}{n}}\right).\notag 
\end{align}
To estimate the last term on the right hand side, note that by the definition on $\vartheta$ and the fact that, for $t>0$, $(1+t)^\beta \leq (1+t)$, the assumption  that $|u_0|+|z_0|\leq m$ implies that
\begin{align}
&{\theta} \left(1 + |z_0|^p \right)  \,\vartheta\left(|u_0|, 2cr^2 \left(\dashint_{B_r }   |\nabla u - z_0|^{\frac{2n}{n+2}}    \, \dd x \right)^{\frac{n+2}{n}}\right)\notag \\
 \leq  & \,c(m,{\theta}) r^{2\beta} \left(\dashint_{B_r }   |\nabla u - z_0|^{\frac{2n}{n+2}}    \, \dd x \right)^{\frac{n+2}{n}\beta} \label{IVRH} \\
\leq &  \, c(m,{\theta}) r^{2\beta} \left(1+\left(\dashint_{B_r }   |\nabla u - z_0|^{\frac{2n}{n+2}}    \, \dd x \right)^{\frac{n+2}{n}}\right)^\beta  \notag \\
\leq & \,c(m, {\theta} )  R_0^{2\beta} \left(1+\left(\dashint_{B_r }   |V(\nabla u - z_0)|^{\frac{2n}{n+2}}    \, \dd x \right)^{\frac{n+2}{n}}\right).\notag 
\end{align}
From  \eqref{Cacc1},  \eqref{IVRH} and $B_r\subseteq B_{R_0}$, $R_0\in(0,1)$, we finally obtain the following estimate:
\begin{align}
&  \dashint_{B_{\zeta r}} |V(\nabla u -z_0)|^2 \,\dd x \notag \\
\leq &\, {\theta}\, \dashint_{B_r }  |V( \nabla u - z_0) |^2  \, \dd x 
 + c \, \left( \dashint_{B_r } \left | V\left(\nabla u - z_0 \right) \right|^\frac{2n}{n+2} \, \dd x\right)^{\frac{n+2}{n}} \notag \\
& + {\theta}\, \dashint_{B_r}  \vartheta( |u_0|, 2 R_0^2 + 2 |u-u_0|^2  ) \left(1 + |\nabla u|^p + |z_0|^p\right) \, \dd x + c \,R_0^{2\beta} .\notag
\end{align}
We can whereby apply Gehring's Lemma from Theorem \ref{theo:Gehring} to obtain that there exists $\varepsilon_1=\varepsilon_1(n,p,\theta,m)>0$ such that, for ${q}_1\in[2,2+\varepsilon_1)$ and a constant $c=c(n,\theta,m)$, 
\begin{align}
&  \left( \dashint_{B_{\frac{R_0}{2}}} |V(\nabla u -z_0)|^{{q_1}} \,\dd x \right)^{\frac{2}{q_1}} \notag \\
\leq &\, c\,\dashint_{B_{R_0} }  |V( \nabla u - z_0) |^2  \, \dd x  \notag \\
& + c\, \left( \dashint_{B_{R_0}}  \vartheta^{\frac{q_1}{2}}( |u_0|, 2 R_0^2 + 2 |u-u_0|^2  ) \left(1 + |\nabla u|^{\frac{pq_1}{2}} \right) \, \dd x \right)^{\frac{2}{q_1}} + c  \,R_0^{2\beta} \label{HintbeforeprelHint}\\
\leq & \, c\,\dashint_{B_{R_0} }  |V( \nabla u - z_0) |^2  \, \dd x  \notag  \\
& + c\, \left( \dashint_{B_{R_0}}  \vartheta( |u_0|, 2 R_0^2 + 2 |u-u_0|^2  ) \left(1 + |\nabla u|^{\frac{pq_1}{2}} \right) \, \dd x \right)^{\frac{2}{q_1}} + c  \,R_0^{2\beta} .\notag
\end{align}
For the last inequality we have used that, since $\vartheta\leq 2L$, then for $\eta>1$, $\vartheta^\eta\leq c(L,\eta)\vartheta$. 
\\
We now take $\varepsilon_0>0$ as the one given by Lemma \ref{PrelHInt} and fix $q_1\in (2,2+\varepsilon_1)$ such that  $\frac{pq_1}{2}\in (p,p+\varepsilon_0 )$. Then, we further choose $q_2\in \left( \frac{pq_1}{2},p+\varepsilon_0   \right)$.  

Using the exponent $q_3=\frac{2q_2}{pq_1}>1$, we now apply H\"{o}lder's inequality to the second term of \eqref{HintbeforeprelHint}. This, and $\vartheta^{\frac{q_3}{q_3-1}}\leq c(L,q_3)\vartheta$, lead to 
\begin{align}
&  \left( \dashint_{B_{\frac{R_0}{2}}} |V(\nabla u -z_0)|^{{q_1}} \,\dd x \right)^{\frac{2}{q_1}} \notag \\
\leq & \, c\,\dashint_{B_{R_0} }  |V( \nabla u - z_0) |^2  \, \dd x  \notag  \\
& + c\,  \left(  \dashint_{B_{R_0}}  \vartheta( |u_0|, 2 R_0^2 + 2 |u-u_0|^2  )\,\dd x \right)^{\frac{2(q_3-1)}{q_1q_3}}
\left( \dashint_{B_{R_0} } \left(1 + |\nabla u|^{q_2} \right) \, \dd x  \right)^{\frac{p}{q_2}}  + c  \,R_0^{2\beta} .\notag
\end{align}
From this point onwards we shall simplify the notation by removing the dependence of $\vartheta$ on $|u_0|$. Now, we use that $\vartheta$ is concave and apply Lemma \ref{PrelHInt} to the right hand side above.  Whereby, we conclude that
\begin{align}
&  \left( \dashint_{B_{\frac{R_0}{2}} } |V(\nabla u -z_0)|^{{q_1}} \,\dd x \right)^{\frac{2}{q_1}} \notag \\
\leq & \, c\,\dashint_{B_{R_0} }  |V( \nabla u - z_0) |^2  \, \dd x  \notag  \\
& + c\, \vartheta \left( 2 R_0^2 +  2\dashint_{B_{R_0}}   |u-u_0|^2  \,\dd x \right)^{\frac{2(q_3-1)}{q_1q_3}}
\left( \dashint_{B_{2R_0} } \left(1 + |\nabla u|^{p} \right) \, \dd x  \right)  + c  \,R_0^{2\beta} \\\notag
\leq & \, c\,\dashint_{B_{R_0} }  |V( \nabla u - z_0) |^2  \, \dd x  \notag  \\
& + c\, \vartheta \left( 2 R_0^2 +  2^{n+1}\dashint_{B_{2R_0}}   |u-u_0|^2  \,\dd x \right)^{\frac{2(q_3-1)}{q_1q_3}}
\left( \dashint_{B_{2R_0} } \left(1 + |\nabla u|^{p} \right) \, \dd x  \right)  + c  \,R_0^{2\beta}. \notag
\end{align}
Finally, we apply Poincar\'e inequality to the argument of $\vartheta$. It is at this point that we use that $u_0=(u)_{2R_0}$. This, together with the fact that $|z_0|\leq m$, lead to the following estimate, for a constant $c>0$:
\begin{align}
&  \left( \dashint_{B_{\frac{R_0}{2}} } |V(\nabla u -z_0)|^{{q_1}} \,\dd x \right)^{\frac{2}{q_1}} \notag \\
\leq & \, c\,\dashint_{B_{R_0} }  |V( \nabla u - z_0) |^2  \, \dd x  \notag  \\
& + c\, \vartheta \left(  2 R_0^2 + 2^{n+1}R_0^2\,\dashint_{B_{2R_0}}  |\nabla u |^2  \,\dd x \right)^{\frac{2(q_3-1)}{q_1q_3}}
\left( \dashint_{B_{2R_0} } \left(1 + |\nabla u  |^{p} \right) \, \dd x  \right)  + c  \,R_0^{2\beta} \notag\\
\leq & \, c\,\dashint_{B_{R_0} }  |V( \nabla u - z_0) |^2  \, \dd x  \notag  \\
& + c\, \vartheta \left(  c R_0^2 + c R_0^2\,\dashint_{B_{2R_0}}  |\nabla u - z_0|^2  \,\dd x \right)^{\frac{2(q_3-1)}{q_1q_3}}
\left( \dashint_{B_{2R_0} } \left(1 + |\nabla u -z_0 |^{p} \right) \, \dd x  \right)  + c  \,R_0^{2\beta}. \notag
\end{align}
The desired conclusion follows after rescaling the constant $c>0$, as well as using that $\vartheta\leq 2L$ and $|\cdot|^p\leq |V(\cdot)|^2$. Indeed, we obtain that
\begin{align}
 &\left( \dashint_{B_{\frac{R_0}{2}} } |V(\nabla u -z_0)|^{{q_1}} \,\dd x \right)^{\frac{2}{q_1}} \notag\\
  \leq & \, c\,\dashint_{B_{2 R_0} }  |V( \nabla u - z_0) |^2  \, \dd x   + c\, {\vartheta}_0 \left(  R_0^2 + R_0^2\,\dashint_{B_{2 R_0}}  |\nabla u - z_0|^2  \,\dd x \right)
  + c  \,R_0^{2\beta}. \notag
\end{align}
  For the last inequality we have  defined ${\vartheta}_0(t):=  \vartheta(ct)^{\kappa}$, with $\kappa=\frac{2(q_3-1)}{q_1q_3}\in (0,1)$. Note that $\kappa=\kappa(n,p,m,L)$, since $\varepsilon_0$ and $\varepsilon_1$ depend on those parameters. This completes the proof of the theorem. 
\end{proof}

\subsection{Linearization of the problem}\label{Sec:Linearization}
Let $(x_0,u_0)\in\Omega\times\R^N$ be fixed. We define the \textit{frozen} integrand $F^0\colon\R^{N\times n}\to\R$ by
\[
F^0(z):=F(x_0,u_0,z).
\]
Furthermore, we consider the second order Taylor polynomial of $F^0$ centred at $z_0\in\R^{N\times n}$, which is given by
\begin{equation}\label{eq:polynomial}
P(z):=F^0(z_0)+F^0_z(z_0)[z-z_0]+\frac{1}{2}F^0_{zz}(z_0)[z-z_0,z-z_0].
\end{equation}
We record here that,  for every $z,w\in\R^{N\times n}$, 
\begin{equation}\label{eq:TpolynomialP}
P(z):=P(w)+P'(w)[z-w]+\frac{1}{2}F^0_{zz}(z_0)[z-w,z-w].
\end{equation}
We then have the following well known approximation estimate:
\begin{lemma}\label{lemmaF-P}
Assume that $F$ satisfies the conditions \ref{H0}-\ref{H4:HoldCont} for some $p\geq 2$. Let $m>1$ be fixed and assume that $|u_0|+|z_0| \leq m$. There exists a modulus of continuity $\omega\colon[0,\infty)\to[0,1]$ such that $\omega$ is increasing, concave, $\omega(t)\geq 1$ for every $t\geq 1$, $\lim_{t\to 0}\omega(t)=0$, and with the property that, for a constant $c=c(m)>0$ and for every $z\in\R^{N\times n}$, 
\begin{equation}\label{eq:F0-Plemma}
|F^0(z)-P(z)|\leq c\,\omega(|z-z_0|^2)|V(z-z_0)|^2.
\end{equation}
\end{lemma}
This result is well known in regularity theory and we shall only make a brief comment regarding its proof.  
\begin{proof}
The construction of the function $\omega$ was already sketched in the proof of Lemma \ref{lemma:modL}. 

In order to obtain  \eqref{eq:F0-Plemma}, it is enough to use Taylor's approximation theorem with the integral form of the remainder for the case in which $|z-z_0|\leq m+1$. On the other hand, if $|z-z_0|>m+1>1$, the proof follows by the truncation technique of Acerbi \& Fusco in \cite{AF87}. 
\end{proof}

The following result is the key linearization strategy that will enable us to obtain partial regularity of $u$. We note that, unlike in the previous results, we assume here that $q\in[2,\infty]$. We need this assumption in order to use the $\LL^p$-estimates stated in Theorem \ref{Lpestimates}. We remark that, just as in \cite{KT}, the case $q\in[1,2)$ should be obtained by establishing first that $u$ satisfies a global minimality property in small subsets of $\Omega$. 
\begin{theorem}[Application of the linearization strategy] \label{theo:linearization}
Assume that $F$ satisfies the conditions \ref{H0}-\ref{H4:HoldCont} for some $p\geq 2$. Let $u\in\WW^{1,p}(\Omega,\R^N)\cap \WW^{1,q}_{\mathrm{loc}}(\Omega,\R^N)$ be a $\WW^{1,q}$-local minimizer of $\mathcal{F}$, where  $q\in[2,\infty]$.

If $q\in[2,\infty)$, then there exists $R_1\in (0,\tfrac{R_\delta}{4})$ such that, for every $\varepsilon\in(0,\frac{1}{4^n})$ and every $m>1$, there exists $\gamma\in (m,n,L,p)\in(0,1)$ such that,     if $\varrho\in(0,R_1)$, $B(x_0,\varrho)\subseteq\Omega$, $|u_{x_0,\varrho}|+|(\nabla u)_{x_0,\varrho}|\leq m$, and $E(x_0,\varrho)<\varepsilon$, then it holds that for every $\tau\in(0,\frac{1}{4})$, 
\begin{equation}\label{Theo:decayI}
E(x_0,\tau \varrho)\leq c\tau^{-n}\varrho^{2\gamma}+c(\tau^{-n}\omega_1(\varepsilon)+\tau^2)E(x_0,\varrho),  
\end{equation}
where $\omega_1(t):=\omega^{\varsigma}(ct)$ for some $\varsigma\in(0,1)$, $c>0$ is a constant,  and both $c$ and $\varsigma$ are depending on $m,n,L, F''$ and $p$. 

On the other hand, if $q=\infty$, we assume in addition that \ref{H0b:C2} and \ref{H1b:Fu} hold for $F$, and that  $u$ satisfies  \eqref{H5:SecVar}. Then, there exists $\delta_1>0$ such that, if \eqref{eq:q=infty} is satisfied, the decay stated in \eqref{Theo:decayI} will remain to be true. 
\end{theorem}

\begin{proof}
We let $R>0$ and assume that $0<4R<R_\delta$. Let $B(x_0,4R)\subseteq\Omega$. 
Furthermore, we denote
\[
u_0:= (u)_{x_0,4R}\,\,\,\mbox{ and }\,\,\,z_0:=(\nabla u)_{x_0,4R}
\]
in the polynomial $P$ defined in \eqref{eq:polynomial} and suppose that $|u_0|+|z_0|\leq m$ for a given $m>1$. Assume also that $E(x_0,4R)<\varepsilon$, with $\varepsilon\in(0,\tfrac{1}{4^n})$.\\
Let $h\in\WW^{1,p}_u(B_{R},\R^N)$ be a minimizer of the functional given by
\[
\mathpzc{P}(v):=\int_{B_R}\! P(\nabla v)\,\dd x.
\]
Then, by the strong quasiconvexity of $F^0$, 
\begin{align}
\dashint_{B_R}\! |V(\nabla u-\nabla h)|^2\,\dd x\leq &\, \dashint_{B_R}\!  \left( F^0(\nabla u-\nabla h+z_0) - F^0(z_0) \right)\,\dd x \notag\\
= &\, \dashint_{B_R}\!  \left( F^0(\nabla u-\nabla h +z_0) - P(\nabla u -\nabla h +z_0) \right)\,\dd x \label{eqDu-Dh} \\
& + \frac{1}{2}\, \dashint_{B_R}\! F^0_{zz} (z_0)[\nabla u -\nabla h,\nabla u -\nabla h]\,\dd x.\notag  
\end{align}
We first note that, since $u-h\in\WW^{1,p}_0(B_R,\R^N)$, and $h$ is a $P$-extremal, \eqref{eq:TpolynomialP} implies that
\begin{align}
& \frac{1}{2}\, \dashint_{B_R}\! F^0_{zz} (z_0)[\nabla u -\nabla h,\nabla u -\nabla h]\,\dd x \notag\\
= & \,\dashint_{B_R}\!  \left(P(\nabla u)-P(\nabla h) \right) \,\dd x \notag\\
= & \,\dashint_{B_R}\! \left( P(\nabla u) - F^0(\nabla u )  \right) \,\dd x + \dashint_{B_R}\! \left( F^0(\nabla u)-F(x,u,\nabla u) \right) \,\dd x \label{II+III+IV+V+VI}\\
& + \dashint_{B_R}\! \left( F(x,u,\nabla u) - F(x,h,\nabla h) \right) \,\dd x  + \dashint_{B_R}\! \left(  F(x,h,\nabla h) - F^0(\nabla h)  \right) \,\dd x \notag\\
& + \dashint_{B_R}\! \left( F^0 (\nabla h) - P(\nabla h) \right) \,\dd x .\notag
\end{align}
From \eqref{eqDu-Dh} and \eqref{II+III+IV+V+VI}, it follows that
\begin{align}
&  \,\dashint_{B_R}\! |V(\nabla u-\nabla h)|^2\,\dd x \notag\\
\leq  & \, \dashint_{B_R}\!  \left( F^0(\nabla u-\nabla h+z_0) - P(\nabla u - \nabla h + z_0) \right)\,\dd x \notag\\
& + \,\dashint_{B_R}\! \left( P(\nabla u) - F^0(\nabla u )  \right) \,\dd x  \label{I+II+III+IV+V+VI}\\
& + \dashint_{B_R}\! \left( F^0(\nabla u)-F(x,u,\nabla u) \right) \,\dd x + \dashint_{B_R}\! \left( F(x,u,\nabla u) - F(x,h,\nabla h) \right) \,\dd x  \notag \\
& + \dashint_{B_R}\! \left(  F(x,h,\nabla h) - F^0(\nabla h)  \right) \,\dd x  + \dashint_{B_R}\! \left( F^0 (\nabla h) - P(\nabla h) \right) \,\dd x \notag\\
= & \, \mathrm{I+II+III+IV+V+VI}. \notag
\end{align}
Observe here that if $q\in[2,\infty)$, inequality \eqref{eq:F''qc} and Theorem \ref{Lpestimates} imply that there exists a constant $K=K( m,n,q)>0$ such that
\[
\|\nabla u-\nabla h\|_{\LL^q(B_R,\R^{N\times n})}\leq \|\nabla u\|_{\LL^q(B_R,\R^{N\times n})}+ \|\nabla h\|_{\LL^q(B_R,\R^{N\times n})}\leq K  \|\nabla u\|_{\LL^q(B_R,\R^{N\times n})}.
\]
Since $u\in\WW^{1,q}_\mathrm{loc}(\Omega,\R^N)$, this implies that, for $R_1\in(0,\tfrac{R_\delta}{4})$ small enough we will have that for every $R\in(0,R_1)$, 
$$\|\nabla u-\nabla h\|_{\LL^q(B_R,\R^{N\times n})}<\delta
$$
 and hence, by the local minimality condition, 
\begin{equation}\label{IV}
\mathrm{IV}<0.
\end{equation}
On the other hand, if $q=\infty$ Theorem \ref{Lpestimates} implies that 
\[
[\nabla h-z_0]_{\mathrm{BMO}(B_R,\R^{N\times n})}\leq A_\infty\|\nabla u-z_0\|_{\LL^{\infty}(B_R,\R^{N\times n})}.
\]
Furthermore, since $u$ is a weak local minimizer, our assumptions enable us to apply Theorem \ref{theo:morethanWLMVMO} to conclude that $u$ is a $\WW^{1,\mathrm{BMO}}$-local minimizer, meaning that for some $\delta_*\in(0,1)$ we will have that, if $h-u\in\WW^{1,2p}_0(\Omega,\R^N)$ and $[\nabla h-\nabla u]_{\mathrm{BMO}(\Omega,\R^{N\times n})}<\delta_*$, then \eqref{IV} will also hold. 

Whereby, by choosing  $\delta_1>0$ such that $\delta_1<\min\{\tfrac{\delta_*}{A_\infty+1},\delta_0\}$, where $\delta_0>0$ is taken as in Lemma \ref{remark:varphi}, assumption \eqref{eq:q=infty} will imply that,  if $R_1>0$ is small enough and $R\in(0,R_1)$, then
\begin{align*}
[\nabla h-\nabla u]_{\mathrm{BMO}(B_R,\R^{N\times n})}\leq &  [\nabla h-z_0]_{\mathrm{BMO}(B_R,\R^{N\times n})}+[\nabla u-z_0]_{\mathrm{BMO}(B_R,\R^{N\times n})}\\
\leq &   A_\infty\|\nabla u - z_0\|_{\LL^\infty(B_R,\R^{N\times n})}+[\nabla u]_{\mathrm{BMO}(B_R,\R^{N\times n})}\\
<& A_\infty\delta_1+\delta_1<\delta_*.
\end{align*}
This means that, for our choice of $\delta_1$, the $\WW^{1,\mathrm{BMO}}$-local minimality can be applied and, whereby, we have \eqref{IV} for every $q\in[2,\infty]$.

To estimate $\mathrm{III}$ and $\mathrm{V}$ we will make use of \ref{H4:HoldCont}. We take $q_0\in (p,p+\varepsilon_0)$ as in Lemma \ref{PrelHInt} and apply H\"{o}lder's inequality to obtain the following:
\begin{align*}
\mathrm{III} \leq & \, c\, \dashint_{B_R}\! \vartheta\left( |x-x_0|^2 + |u-u_0|^2 \right) \left(1+|\nabla u|^p  \right) \,\dd x \notag\\
\leq & \,c \left(   \dashint_{B_R} \vartheta^{\frac{{q_0}}{q_0 -p}} \left(R^2 + |u-u_0|^2      \right) \,\dd x\right)^{\frac{ q_0 -p}{q_0 }} \left( \dashint_{B_R} \left(1 + |\nabla u|^p \right)^{\frac{q_0 }{p}} \,\dd x \right)^\frac{p}{q_0 }\notag\\
\leq & \,c\,\vartheta\left(R^2 + \dashint_{B_{4R}} |u-u_0|^2   \,\dd x  \right)^{\frac{q_0 -p}{q_0 }} \left( \dashint_{B_R} \left(1 + |\nabla u|^{q_0} \right) \,\dd x  \right)^\frac{p}{q_0}.
\end{align*}
The last inequality follows after increasing the size of the ball of integration and, then, using that $\vartheta\leq 2L$ is concave. 
\\
We now apply Poincar\'e inequality and the higher integrability from Lemma \ref{PrelHInt}, which for $q=\infty$ we can do because $\delta_1<\delta_0$. This allows us to  obtain that, for some constant $c=c(m,n)>0$ and for $\vartheta_1(t):=\vartheta(ct)^\frac{q_0 -p}{q_0 }$, 
\begin{align}
\mathrm{III} \leq & \,c\, \vartheta\left(R^2 + R^2\dashint_{B_{4R}} |\nabla u|^2   \,\dd x  \right)^{\frac{q_0-p}{{q_0}}}  \dashint_{B_{2R}} \left(1 + |\nabla u|^{p} \right) \,\dd x   \notag\\
\leq & \,c\, \vartheta_1\left(R^2 + R^{2}\dashint_{B_{4R}} |\nabla u - z_0|^2   \,\dd x  \right)  \dashint_{B_{4R}} \left(1 + |\nabla u-z_0|^{p} \right) \,\dd x.   \label{III}
\end{align}
We note that we have made $R=\frac{R_0}{2}$ while applying Lemma \ref{PrelHInt}, which is consistent with our assumption that $|u_0|+|z_0|\leq m$. 
\\
We move on to estimate $\mathrm{V}$. By applying assumption \ref{H4:HoldCont} and then H\"{o}lder's inequality, we have that
\begin{align*}
\mathrm{V} \leq & \, c\, \dashint_{B_R}\! \vartheta\left( |x-x_0|^2 + |h-u_0|^2 \right) \left(1+|\nabla h|^p  \right) \,\dd x \notag\\
 & + c\, \dashint_{B_R}\! \vartheta\left( R^2 + 2|h-u|^2 + 2|u-u_0|^2 \right) \left(1+|\nabla h|^p  \right) \,\dd x \notag\\
\leq & \,c \left(   \dashint_{B_R} \vartheta^{\frac{{q_0 }}{q_0 -p}} \left(R^2 + 2|h-u|^2 + 2|u-u_0|^2      \right) \,\dd x\right)^{\frac{q_0 -p}{q_0 }} \left( \dashint_{B_R} \left(1 + |\nabla h|^p \right)^{\frac{q_0 }{p}} \,\dd x \right)^\frac{p}{q_0 }\notag\\
\leq & \,c\,\vartheta\left(R^2 + \dashint_{B_R} \left(2|h-u|^2 + 2|u-u_0|^2  \right)  \,\dd x  \right)^{\frac{q_0 -p}{q_0 }} \left( \dashint_{B_R} \left(1 + |\nabla h|^{q_0 } \right) \,\dd x  \right)^\frac{p}{q_0 }\notag\\
\leq & \,c\,\vartheta_1\left(R^2 + \dashint_{B_R} |h-u|^2\,\dd x + \dashint_{B_{4R}}|u-u_0|^2    \,\dd x  \right) \left( \dashint_{B_R} \left(1 + |\nabla h|^{q_0 } \right) \,\dd x  \right)^\frac{p}{q_0 },\notag
\end{align*}
where the constant $c>0$ that appears in the definition of $\vartheta_1$, may have changed. 

As before, we now apply Poincar\'e inequality to the functions $h-u\in\WW^{1,p}_0(B_R,\R^N)$ and $u-u_0$, to estimate the first term on the right hand side. This, together with the $\LL^p$-estimates for $h$ from \eqref{Lpestimates}, leads to 
\begin{align*}
\mathrm{V} \leq & \,c\vartheta_1\left(R^2 + R^2\dashint_{B_R} |\nabla h-\nabla u|^2 \,\dd x+R^2\dashint_{B_{4R}}  |\nabla u|^2   \,\dd x  \right) \left( \dashint_{B_R} \left(1 + |\nabla u|^{q_0 } \right) \,\dd x  \right)^\frac{p}{q_0 } \notag\\
\leq &\, c\vartheta_1\left(R^2 + R^2\dashint_{B_{4R}} \left(|\nabla h-z_0|^2 + |\nabla u-z_0|^2  + |\nabla u|^2  \right)  \,\dd x  \right) \left( \dashint_{B_R} \left(1 + |\nabla u|^{q_0 } \right) \,\dd x  \right)^\frac{p}{q_0 }\hspace{-1.5mm}. \notag
\end{align*}
For the last inequality above, the argument of $\vartheta_1$ has been rescaled again without changing the notation. 
\\
By Theorem \ref{Lpestimates},  Lemma \ref{PrelHInt}, and the fact that $|z_0|\leq m$, we can further obtain that
\begin{align}
\mathrm{V} \leq & \,c\,\vartheta_1\left(R^2 + 2 R^2\dashint_{B_{4R}} |\nabla u- z_0 |^2    \,\dd x  \right) \dashint_{B_{2R}} \left(1 + |\nabla u-z_0|^{p} \right) \,\dd x \notag  \\
\leq & \,c\,\vartheta_1 \left(R^2 +  R^2\dashint_{B_{4R}}  |\nabla u-z_0|^2    \,\dd x  \right) \dashint_{B_{4R}} \left(1 + |\nabla u-z_0|^{p} \right) \,\dd x. \label{V} 
\end{align}
We now  estimate $\mathrm{II+VI}$. Indeed, applying Lemma \ref{lemmaF-P} and H\"{o}lder's inequality we obtain that, for $q_1\in (2,2+\varepsilon_1)$ as in Theorem \ref{theo:ReverseHolder}, the following holds:
\begin{align*}
\mathrm{II+VI} \leq & \,c\,\dashint_{B_R}\omega(|\nabla u - z_0|^2)|V(\nabla u - z_0)|^2\,\dd x + c\,\dashint_{B_R}\omega(\nabla h - z_0)|V(\nabla h - z_0)|^2\,\dd x \notag \\
\leq &\,c\,\left(\dashint_{B_R} \omega^{\frac{q_1}{q_1-2}} (|\nabla u -z_0|^2)\right)^{\frac{q_1-2}{q_1}}  \left( \dashint_{B_R} |V(\nabla u - z_0)|^{q_1} \right)^{\frac{2}{q_1}} \notag\\
& + \,c\,\left(\dashint_{B_R} \omega^{\frac{q_1}{q_1-2}} (|\nabla h -z_0|^2)\right)^{\frac{q_1-2}{q_1}}  \left( \dashint_{B_R} |V(\nabla h - z_0)|^{q_1} \right)^{\frac{2}{q_1}} \notag\\
\leq & \,c\,\omega\left(\dashint_{B_R} |\nabla u -z_0|^2\right)^{\frac{q_1-2}{q_1}}  \left( \dashint_{B_R} |V(\nabla u - z_0)|^{q_1} \right)^{\frac{2}{q_1}} \notag\\
& + \,c\,\omega \left(\dashint_{B_R} |\nabla h -z_0|^2 \right)^{\frac{q_1-2}{q_1}}  \left( \dashint_{B_R} |V(\nabla h - z_0)|^{q_1} \right)^{\frac{2}{q_1}}. \notag
\end{align*}
For the last inequality we have used, once again, that $\omega\leq 1$ is concave. Using this one more time, as well as the $\LL^p$-estimates from Theorem \ref{Lpestimates}, we derive from above that
\begin{align*}
\mathrm{II+VI} \leq &  \,c\,\omega\left(c\,\dashint_{B_R} |\nabla u -z_0|^2\right) ^{\frac{q_1-2}{q_1}} \left( \dashint_{B_R} |V(\nabla u - z_0)|^{q_1} \right)^{\frac{2}{q_1}} . \notag
\end{align*}
Therefore, by \eqref{eq:RevHold}, we obtain that for $\vartheta_1$ as it was previously defined, 
\begin{align}
 \mathrm{II+VI} \leq &  \,c\,\omega\left(c\,\dashint_{B_{4R}} |\nabla u -z_0|^2\right)^{\frac{q_1-2}{q_1}} \dashint_{B_{4 R} }  |V( \nabla u - z_0) |^2  \, \dd x  \notag\\
 & +\,c\,\omega\left(c\,\dashint_{B_{4R}} |\nabla u -z_0|^2\right)^{\frac{q_1-2}{q_1}}  {\vartheta}_1 \left(  R^2 + R^2\,\dashint_{B_{4R}}  |\nabla u - z_0|^2  \,\dd x \right) 
\label{II+VI}\\
 &  + c \,\omega\left(c\,\dashint_{B_{4R}} |\nabla u -z_0|^2\right)^{\frac{q_1-2}{q_1}} \,R^{2\beta}.\notag
\end{align}
Finally, we will now estimate term $\mathrm{I}$. In what follows, we first use Lemma \ref{lemmaF-P} and then we take $q_1\in (2,2+\varepsilon)$ as in the estimate of $\mathrm{II+IV}$. We use H\"{o}lder's inequality and, once again, exploit the fact that $\omega\leq 1 $ is concave to obtain that
\begin{align}
\mathrm{I} \leq & \, c\,\dashint_{B_R}\omega(|\nabla u - \nabla h|^2)|V(\nabla u - \nabla h)|^2\,\dd x\notag\\
\leq  & \,c\,\omega\left( \dashint_{B_R}|\nabla u - \nabla h|^2\,\dd x\right)^{\frac{q_1-2}{q_1}}  \left( \dashint_{B_R}|V(\nabla u - \nabla h)|^{q_1} \,\dd x \right)^{\frac{2}{q_1}} . \notag
\end{align}
Now, using Theorem \ref{Lpestimates}, the previous inequality leads to
\begin{align}
\mathrm{I} \leq & \,c\,\omega\left(c\, \dashint_{B_R}\left( |\nabla u - z_0|^2 + |\nabla h - z_0|^2 \right) \,\dd x\right)^{\frac{q_1-2}{q_1}}  \left( \dashint_{B_R}|V(\nabla u - z_0)|^{q_1} \,\dd x   \right)^{\frac{2}{q_1}} \notag \\
& +\,c\,\omega\left(c\, \dashint_{B_R}\left( |\nabla u - z_0|^2 + |\nabla h - z_0|^2 \right) \,\dd x\right) ^{\frac{q_1-2}{q_1}} \left( \dashint_{B_R}|V(\nabla h - z_0)|^{q_1} \,\dd x   \right)^{\frac{2}{q_1}} \notag\\
\leq & \,c\,\omega\left(c \,\dashint_{B_R}|\nabla u - z_0|^2\,\dd x\right) ^{\frac{q_1-2}{q_1}} \left( \dashint_{B_R}|V(\nabla u - z_0)|^{q_1} \,\dd x  \right)^{\frac{2}{q_1}}. \notag 
\end{align}
Finally, we let $\omega_1(t):=\omega(c\,t)^{\frac{q_1-2}{q_1}} $ and note that, by the higher integrability from Theorem \ref{theo:ReverseHolder}, this leads to the following estimate:
\begin{align}
\mathrm{I} \leq & \,c\,\omega_1\left(c \,\dashint_{B_{4R}}|\nabla u - z_0|^2\,\dd x\right)  \dashint_{B_{4R}}|V(\nabla u - z_0)|^{2} \,\dd x  \notag\\
& +  \,c\,\omega_1\left(c \,\dashint_{B_{4R}}|\nabla u - z_0|^2\,\dd x\right)  {\vartheta}_1 \left(  R^2 + R^2\,\dashint_{B_{4R}}  |\nabla u - z_0|^2  \,\dd x \right)  \label{I} \\
& + c \,\omega_1\left(c\,\dashint_{B_{4R}} |\nabla u -z_0|^2\right) \,R^{2\beta}.\notag
\end{align}
Since $\omega_1\leq 1$ and $\vartheta_1\leq c(m,n,p,L)$, it follows from the estimates \eqref{IV}-\eqref{I} that, if $R\in (0,R_1)$, $u_0=(u)_{4R}$ and  $z_0=(\nabla u)_{4R}$ satisfy $|u_0|+|z_0|\leq m$ and $E(x_0,4R)<\varepsilon<1$, then for a constant $c=c(m)>0$, 
\begin{align}
\mathrm{I+II+III+IV+V+VI}\leq & \,c\, \vartheta_1(2R^2)+cR^{2\beta}+\omega_1(\varepsilon)E(x_0,4R)\notag\\
\leq & \,c \,R^{2\beta\kappa}+\omega_1(\varepsilon)E(x_0,4R).\label{I-Vfinal}
\end{align}
The last inequality is making use of the definition of $\vartheta_1$ and the fact that $\kappa<1$. 
\\
For the last step of the proof we note that, by  Lemma \ref{lemma:qminofMean}, inequality \eqref{eq:F''qc}, and Theorem \ref{theo:DecayAHarmonic}, it follows that for every $r\in (0,\frac{R}{4})$, 
\begin{align}
\dashint_{B_r} |V(\nabla u - (\nabla u)_r)|^2\,\dd x \leq &\,c\, \dashint_{B_r} |V(\nabla u - (\nabla h)_r)|^2\,\dd x \notag\\
\leq & \,c\, \dashint_{B_r} |V(\nabla u - \nabla h )|^2\,\dd x + \,c\, \dashint_{B_r} |V(\nabla h - (\nabla h)_r)|^2\,\dd x \notag \\
\leq & \,c\left(\frac{R}{r} \right)^n \dashint_{B_R} |V(\nabla u - \nabla h )|^2\,\dd x \label{FinalEstLin} \\
& +  \,c\,\left(\frac{r}{R}\right)^2 \dashint_{B_R} |V(\nabla u - (\nabla u)_R)|^2\,\dd x.\notag
\end{align}
Combining the estimates \eqref{I+II+III+IV+V+VI}, \eqref{I-Vfinal} and \eqref{FinalEstLin}, and increasing again the size of the ball $B_R$ in \eqref{FinalEstLin}, we can conclude that 
\begin{align}
E(x_0,r) \leq & \,c\left(\frac{R}{r} \right)^n   R^{2\beta\kappa}+c\left( \left(\frac{R}{r} \right)^n   \omega_1(\varepsilon) + \left(\frac{r}{R}\right)^2 \right) E(x_0,4R) .\notag
\end{align}
By writing $\varrho=4R$, we can reinterpret this as the fact that, for every $\varepsilon\in(0,\frac{1}{4^n})$ and $m>1$, if $\varrho\in(0,4R_1)$, $B(x_0,\varrho)\subseteq\Omega$, $|(u)_{x_0,\varrho}|+|(\nabla u)_{x_0,\varrho}|<m$, and $E(x_0,\varrho)<\varepsilon$, then for every $\tau\in(0,\frac{1}{4})$, 
\begin{equation*}
E(x_0,\tau \varrho)\leq c\tau^{-n}\varrho^{2\beta\kappa}+c(\tau^{-n}\omega_1(\varepsilon)+\tau^2)E(x_0,\varrho). 
\end{equation*}
\end{proof}

\begin{proof}[Proof of Proposition \ref{Prop:FinalDecay}]
The estimate follows readily from \eqref{Theo:decayI}. Indeed, by taking the constant $c>0$ to be larger if necessary, we take $\gamma:=\beta\kappa$, with $\kappa\in(0,1)$ as in Theorem \ref{theo:linearization}, $\alpha\in(2\gamma,1)$, and $\tau_0  \in (0,\frac{1}{4})$ such that
\[
c\tau_0^2=\frac{1}{2}\tau_0^\alpha, 
\]
or, in other words, $\tau_0:=(2c)^{-\frac{1}{2-\alpha}}$. 

Furthermore, for such a $\tau_0=\tau_0(m,n,N,L)$, we impose a further smallness condition on $\varepsilon>0$, so that 
\begin{equation}\label{eq:smallnessepsilon}
c\omega_1(\varepsilon)\tau_0^{-n}\leq \frac{1}{2}\tau_0^\alpha. 
\end{equation}
This is possible because $\omega_1$ is increasing, continuous at $0$ and $\omega_1(0)=0$. Whereby, there exists $\varepsilon_1\in(0,\frac{1}{4^n})$ such that, for every $\varepsilon\in(0,\varepsilon_1)$, \eqref{eq:smallnessepsilon} holds. 

We now let $c_0:= c\tau_0^{-n}$ and conclude that 

\begin{equation*}
E(x_0,\tau_0\varrho) \leq c_0 \varrho^{2\gamma}+\tau_0^\alpha E(x_0,\varrho).
\end{equation*}
This completes the proof of the proposition.
\end{proof}
We conclude by establishing the proof of the regularity result. The proof follows a classical iteration process and it fully relies on Proposition \ref{Prop:FinalDecay}. However, since the exponent $\gamma$ depends on $m>1$ in the previous statement, we include below the main steps of the iteration, for the convenience of the reader. 

\begin{proof}[Proof of Theorem \ref{theo:regularity}]
We begin by fixing $m>0$ and noting that, if $\varepsilon_1>0$ taken as in Proposition \ref{Prop:FinalDecay} for the fixed number $m+1$, and if $0<\varepsilon<\varepsilon_1<1$, then  for any $x_0\in\Omega_0$ and any $\varrho\in\left(0,\frac{1}{2}\mathrm{dist}(x_0,\partial\Omega)\right)$ satisfying that
\[
|(u)_{x_0,\varrho}|+|(\nabla u)_{x_0,\varrho}|<m\,\,\,\mbox{ and }\,\,\,E(x_0,\varrho)<\varepsilon,
\]
there exists $\varrho_0\in (0,\varrho)$ such that, for every $x\in B(x_0,\varrho_0)$, it also holds that
\begin{equation*}
|(u)_{x,\varrho}|+|(\nabla u)_{x,\varrho}|<m\,\,\,\mbox{ and }\,\,\,E(x,\varrho)<\varepsilon.
\end{equation*}
This means that, for every $x\in B(x_0,\varrho)$ and for $0< \tau_0<\frac{1}{4}$ as in Proposition \ref{Prop:FinalDecay}, it holds that 
 \begin{equation*}
E(x_0,\tau_0\varrho) \leq c_0 \varrho^{2\gamma}+\tau_0^\alpha E(x_0,\varrho).
\end{equation*}
A standard iteration process (see, for example, \cite[Proposition 9.4]{Giusti}) leads to finding a constant $M=M(n,m,L,p)>0$ such that, for every $x\in B(x_0,\varrho)$, and for every $r\in \left(0,\frac{1}{4}\varrho\right)$, 
\begin{equation}
E(x,r)\leq M\left(\frac{r}{\varrho}\right)^{\gamma}, 
\end{equation}
where $\gamma$ is also as in Proposition \ref{Prop:FinalDecay}. Therefore, by Campanato-Meyers' classical characterization of H\"{o}lder continuity, it follows that $ u\in \CC^{1,\frac{\gamma}{2}}(B(x_0,\varrho_0),\R^N)$. 

In order to establish that, in fact, $ u\in \CC^{1,{\beta}}(\Omega_0,\R^N)$, with $\beta$ as in \ref{H4:HoldCont}, it is sufficient to go again through the proof of Theorem \ref{theo:linearization} while exploiting the fact that $\nabla u$ is already known to be continuous in $\Omega_0$, in the spirit of using Schauder estimates. 

Indeed, following the notation of Theorem \ref{theo:linearization} we can take $R>0$  so that $\overline{B(x_0,R)}\subseteq\Omega_0$, and $R>0$  small enough so that there exists a constant $M_1=M_1(m)>0$ such that 
\begin{equation}\label{eq:M1}
\| u\|_{\CC^{1,\frac{\gamma}{2}}(\overline{B(x_0,R)},\R^{N})}+\|h\|_{\CC^{1,\frac{\gamma}{2}}(\overline{B(x_0,R)},\R^{N})}\leq M_1. 
\end{equation}
Whereby, revisiting the estimates for each term on the right hand side of \eqref{I+II+III+IV+V+VI}, we note that there is no longer need to apply the higher integrability of $\nabla u$ and H\"{o}lder inequality to estimate $\mathrm{III}$ and $\mathrm{V}$. Instead, we can use \eqref{eq:M1} to conclude that, since $\vartheta$ is concave, then for a constant $c=c(m,M_1,n)>0$, 
\begin{equation}
\mathrm{III}+\mathrm{V}\leq c\,\vartheta\left(cR^2 \right)\leq cR^{2\beta}. 
\end{equation}
In a similar fashion we can obtain that 
\begin{equation}
\mathrm{I}+\mathrm{II}+\mathrm{VI}\leq \omega_1(E(x_0,4R)E(x_0,4R)+ cR^{2\beta},  
\end{equation}
where $c>0$ now also depends on the modulus of continuity $\omega$. 

Compiling and applying these estimates to the corresponding inequality \eqref{FinalEstLin}, we can conclude that, for every $\varepsilon\in(0,\frac{1}{4^n})$ and $m>1$, if $\varrho\in(0,4R_1)$, $B(x_0,\varrho)\subseteq\Omega$, $|(u)_{x_0,\varrho}|+|(\nabla u)_{x_0,\varrho}|<m$, and $E(x_0,\varrho)<\varepsilon$, then for every $\tau\in(0,\frac{1}{4})$, 
\begin{align*}
E(x_0,\tau \varrho)\leq & c\tau^{-n}\varrho^{2\beta}+c(\tau^{-n}\omega_1(\varepsilon)+\tau^2)E(x_0,\varrho).
\end{align*}
Whereby, we can argue again exactly as we did to obtain that $u\in\CC^{1,\frac{\gamma}{2}}(B(x_0,R),\R^N)$ to conclude now that, for some constant $C>0$, and for any $r\in (0,\tfrac{1}{4}\varrho)$, 
\[
E(x_0,r)\leq\ Cr^{2\beta},
\]
with the constant $C$ depending in particular on $\varrho$, on $\beta$ and, once again, on $m>1$. This completes the proof of Theorem \ref{theo:regularity}.  
\end{proof}

\section*{Funding}

This work was supported by Universidad Nacional Aut\'onoma de M\'exico - Direcci\'on General de Asuntos del Personal Acad\'emico - Programa de Apoyo a Proyectos de Investigación e Innovación Tecnológica [IA105221].

\section*{Acknowledgments}
The author wishes to thank Jan Kristensen for inspiring discussions. The author is also grateful to the anonymous referee for providing very valuable corrections and for such a careful review of the manuscript.

\noindent 
Departamento de Matem\'{a}ticas, Facultad de Ciencias, Universidad Nacional Aut\'{o}noma de M\'{e}xico,     Circuito exterior s/n, Ciudad Universitaria, C.P. 04510, Ciudad de M\'{e}xico, M\'{e}xico
\\
e-mail: \textit{judith@ciencias.unam.mx}


\begin{thebibliography}{100}

\bibitem{AF87} E.~Acerbi and N.~Fusco: A regularity theorem for minimizers of quasiconvex
integrals. \textit{Arch. Rational Mech. Anal.}, \textbf{99} (1987), No. 3, 261--281.

\bibitem{AF89} E.~Acerbi and N.~Fusco: Regularity for minimizers of nonquadratic functionals: the case $1 < p < 2$.\textit{ J. Math. Anal. Appl.}, \textbf{140} (1989), No. 1, 115--135.

\bibitem{Al68} F.J.~Almgren, Jr.: Existence and regularity almost everywhere of solutions to elliptic variational problems among surfaces of varying topological type and singularity structure.\textit{ Ann. Math.}  \textbf{87} (1968), No. 2, 321--391.

\bibitem{Al76} F.J.~Almgren, Jr.: Existence and regularity almost everywhere of solutions to elliptic variational problems with constraints. \textit{Memoirs of the American Mathematical Society}, \textbf{4} (1976), No. 154.

\bibitem{Beck2011} L.~Beck: Boundary regularity results for variational integrals. \textit{Q. J. Math.}, \textbf{62} (2011), No. 4, 791--824.

\bibitem{BG17} W.T.~Bitew and Y.~Grabovsky: Higher regularity of uniform local minimizers in Calculus of Variations. \textit{ Proc. Amer. Math. Soc.} \textbf{145} (2017), No. 12, 5215--5222.


\bibitem{JCC2017} J.~Campos~Cordero: Boundary regularity and sufficient conditions for strong local minimizers. \textit{J. Funct. Anal.} \textbf{272} (2017), No. 11, 4513--4587.

\bibitem{JCCKK} J.~Campos~Cordero and K.~Koumatos: Necessary and sufficient conditions for the strong local minimality of extremals on a class of non-smooth domains. \textit{ESAIM: Control, Optimization and Calculus of Variations}, \textbf{26} (2020), No. 49.

\bibitem{JCCJKreg} J.~Campos~Cordero and J.~Kristensen: Uniqueness results under natural smallness conditions. \textit{Preprint,} (2021). 

\bibitem{JCCJKuniq} J.~Campos~Cordero and J.~Kristensen:  Regularity results under smallness conditions. \textit{In preparation.}

\bibitem{CN03} M.~Carozza and A.~Di Napoli: Partial regularity of local minimizers of quasiconvex integrals with sub-quadratic growth. \textit{Proceedings of the Royal Society of Edinburgh: Section A Mathematics}, \textbf{133} (2003), No. 6, 1249--1262. 

\bibitem{CFP18} M.~Carozza, I.~Fonseca, A.~Passarelli di Napoli: Regularity results for an optimal design problem with quasiconvex bulk energies.\textit{ Calc. Var. Partial Differ. Equ}. \textbf{57} (2018), No. 2, 57--68.

\bibitem{CFM98} M.~Carozza, N.~Fusco, and G.~Mingione: Partial regularity of minimizers of quasiconvex integrals with subquadratic growth. \textit{Ann. Mat. Pura Appl.}  \textbf{175} (1998) No. 4, 141--164.

\bibitem{CYCJK} C.Y.~Chen and J.~Kristensen: On coercive variational integrals. \textit{Nonlinear Analysis: Theory, Methods \& Applications,} \textbf{153} (2017), 213--229.

\bibitem{CFLM20} G.~Cupini, M.~Focardi, F.~Leonetti and E.~Mascolo: On the Hölder continuity for a class of vectorial problems. \textit{Advances in Nonlinear Analysis}, \textbf{9} (2020), No. 1, 1008--1025.

\bibitem{Dac} B.~Dacorogna: \textit{Direct methods in the calculus of variations. Second Edition}. Applied Mathematical Sciences, \textbf{78}. Springer, New York, 2008. xii+619pp.

\bibitem{DG61} E.~De Giorgi: Frontiere orientate di misura minima. \textit{Seminario di Matematica della Scuola Normale Superiore di Pisa, 1960--61 Editrice Tecnico Scientifica, Pisa},  (1961), (Italian).



\bibitem{DLSV12} L.~Diening, D.~Lengeler, B.~Stroffolini, and A.~Verde: Partial regularity for minimizers of quasi-convex functionals with general growth. \textit{SIAM J. Math. Anal.} \textbf{44} (2012), No. 5 3594--3616.

\bibitem{DM04} F.~Duzaar and G.~Mingione: Regularity for degenerate elliptic problems via p-harmonic approximation. \textit{Ann. Inst. H. Poincaré Anal. Non Linéaire} \textbf{21} (2004), No. 5, 735--766.

\bibitem{Eva86} L.~C.~Evans: Quasiconvexity and partial regularity in the calculus of variations. \textit{Arch. Rational Mech. Anal.} \textbf{95} (1986) No. 3, 227--252.

\bibitem{EG87} L.~C.~Evans and R.~F.~Gariepy: Blow-up, compactness and partial regularity in the calculus of variations. \textit{Indiana Univ. Math. J.},  \textbf{36} (1987), 361--371.



\bibitem{FH85} N.~Fusco and J.~Hutchinson: $\CC^{1,\alpha}$ partial regularity of functions minimising quasiconvex integrals. \textit{Manuscripta Math.}, \textbf{54}(1985) 121--143.

\bibitem{FH91} N.~Fusco and J.~Hutchinson: Partial regularity in problems motivated by nonlinear elasticity.\textit{ SIAM J. Math. Anal.} \textbf{22} (1991), No. 6, 1516--1551.

\bibitem{Geh73} F.~W.~Gehring: The $\LL^p$-integrability of the partial derivatives of a quasiconformal mapping. \textit{Acta Math.}, \textbf{130} (1973), 265--277.

\bibitem{Gia} M.~Giaquinta: \textit{Multiple integrals in the calculus of variations and nonlinear elliptic systems}. Annals of Mathematics Studies, \textbf{105}, Princeton University Press, (1983).

\bibitem{Gia1988} M.~Giaquinta: Quasiconvexity, growth conditions and partial regularity. 
S. Hildebrandt and R. Leis (Eds.), Partial differential equations and calculus
of variations. Lecture notes in mathematics, Vol. \textbf{1357}, Springer, Berlin, 1988.

\bibitem{GiaqMar} M.~Giaquinta and L.~Martinazzi: \textit{An introduction to the regularity theory for elliptic systems, harmonic maps and minimal graphs.} Springer Science \& Business Media, 2013.

\bibitem{GiaqGius} M.~Giaquinta and E.~Giusti: On the regularity of the minima of variational integrals. Acta Math. 148 (1982),  31--46.

\bibitem{GiaqMod} M.~Giaquinta and G.~Modica: {Regularity results for some classes of higher order nonlinear elliptic systems,} \textit{J. f\"ur reine u angew. Math.} 311/312, (1979) 145--169. 

\bibitem{GiaqMod86} M.~Giaquinta and G.~Modica: Partial regularity of minimizers of quasiconvex integrals. \textit{Annales de l’I. H. P.}, section C, \textbf{3} (1986), No. 3,  185--208.

\bibitem{Giusti} E.~Giusti: \textit{Direct methods in the calculus of variations}. 
World Scientific Publishing Co., Inc., River Edge, NJ, 2003. viii+403 pp.

\bibitem{GiMi} E.~Giusti, M.~Miranda: Sulla regolarit\'a delle soluzioni deboli di una classe di sistemi ellittici quasi-lineari. \textit{Arch. Ration. Mech. Anal.} \textbf{31} (1968/1969), 173--184,  (Italian).

\bibitem{GK19} F.~Gmeineder and J.~Kristensen: Partial Regularity for BV Minimizers. \textit{Arch Rational Mech Anal.} \textbf{232} (2019), 142--1473. 


\bibitem{GM} Y.~Grabovsky and T.~Mengesha: Sufficient conditions for strong local minimal: the case of $C^1$ extremals, \textit{Trans.~Amer.~Math.~Soc.} \textbf{361} (2009), No.~3, 1495--1541.

\bibitem{GrIwaMos} L.~Greco, T.~Iwaniec, and G.~Moscariello: Limits of the improved integrability of the volume forms. \textit{Indiana Univ. Math. J.}, \textbf{44} (1995), No. 2, 305--339.

\bibitem{Hamburger96} C.~Hamburger: Partial regularity for minimizers of variational integrals with discontinuous integrands, \textit{Annales de l'Institut Henri Poincaré C, Analyse non linéaire}  \textbf{13} (1996), No. 3,  255--282. Elsevier Masson.


 
\bibitem{CH16} C.P.~Hopper: Partial regularity for holonomic minimisers of quasiconvex functionals.\textit{ Arch. Ration. Mech. Anal.} \textbf{222} (2016), No. 1, 91--141.

\bibitem{CIrving} C.~Irving: $\mathrm{BMO}$ $\varepsilon$-regularity results for solutions to Legendre-Hadamard elliptic systems. \textit{Preprint,} (2021). 


\bibitem{John} F.~John: Uniqueness of non-linear elastic equilibrium for prescribed boundary displacements and sufficiently small strains. 
\textit{Comm.~Pure Appl.~Math.} \textbf{25} (1972), 617--634.

\bibitem{KM07} J.~Kristensen and G.~Mingione: The singular set of Lipschitzian minima of multiple integrals. \textit{Arch. Ration. Mech. Anal.}, \textbf{184} (2007) No. 2, 341--369.

\bibitem{KT} J.~Kristensen and A.~Taheri: Partial regularity of strong local minimizers in the multi-dimensional calculus of variations. \textit{Arch. Ration. Mech. Anal.}, \textit{170} (2003), No. 1, 63--89.

\bibitem{ME75}  N. G. Meyers and A. Elcrat: Some results on regularity for solutions of non-linear elliptic systems and quasi-regular functions. \textit{Duke Math. J}. \textbf{42} (1975), No.~1, 121--136. 

\bibitem{Morrey} C.B.~Morrey, Jr.: Partial regularity results for non-linear elliptic systems. \textit{J. Math. Mech.} \textbf{17} (1967/1968), 649--670.


\bibitem{MS03} S.~M\"uller and V.~{\v{S}}ver{\'a}k: Convex integration for Lipschitz mappings and counterexamples to regularity. \textit{Ann. Math.} 157,(2003), 715--742.

\bibitem{PS1997} K.D.E~Post and J.~Sivaloganathan: On homotopy conditions and the existence of multiple equilibria in finite elasticity. \textit{Proceedings of the Royal Society of Edinburgh Section A: Mathematics}, \textbf{127} (1997) No. 3 595--614. 

\bibitem{ScSc} S.~Schemm and T.~Schmidt: Partial regularity of strong local minimizers of quasiconvex integrals with $(p, q)$-growth.  \textit{Proceedings of the Royal Society of Edinburgh: Section A Mathematics}, \textbf{139} (2009), No. 3, 595--621.

\bibitem{TSchmidt} T.~Schmidt: Regularity of minimizers of $\WW^{1,p}$-quasiconvex variational integrals with $(p, q)$-growth. \textit{Calc. Var. Partial Differential Equations}, \textbf{32} (2008), No. 1, 1--24.

\bibitem{SpectorSpectorBMO} D.E.~Spector and S.J~Spector: $\mathrm{BMO}$ and Elasticity: Korn’s Inequality; Local Uniqueness in Tension. \textit{J Elast.} \textbf{143} (2021), 85--109.

\bibitem{Stredulinsky} E.W.~Stredulinsky: {Higher integrability from reverse Hölder inequalities.} \textit{Indiana University Mathematics Journal,} \textbf{29} (1980), No. 3, 407--13.

\bibitem{Tah2001a} A.~Taheri: On Artin's braid group and polyconvexity in the calculus of variations. \textit{Journal of the London Mathematical Society}, \textbf{67}  (2003), No. 3, 752--768.


\bibitem{Tah03} A.~Taheri: Quasiconvexity and uniqueness of stationary points in the multi-dimensional calculus of variations. 
\textit{Proc.~Amer.~Math.~Soc.} \textbf{131}(10) (2003), 3101--3107.

\bibitem{Tah05} A.~Taheri: Local Minimizers and Quasiconvexity -- the Impact of Topology. \textit{Arch. Rational Mech. Anal.} \textbf{176} (2005), 363--414.
\\
\\

\end{thebibliography}
\end{document}